\documentclass[11pt]{article}
\usepackage[T1]{fontenc}
\usepackage{amsfonts}
\usepackage{amsmath}
\usepackage{amssymb}
\usepackage{amsthm}
\usepackage{bbm}
\usepackage{bm}
\usepackage{mathrsfs}
\usepackage{color}
\usepackage{pdfsync}
\usepackage{enumitem}

\usepackage{tikz}
\usepackage[left=31 mm,right=31 mm,top=29 mm,bottom=29 mm]{geometry}

\DeclareMathOperator*{\argmax}{argmax}
\DeclareMathOperator*{\essinf}{ess\,inf}
\DeclareMathOperator*{\esssup}{ess\,sup}
\newcommand{\R}{\mathbb{R}}

\newcommand{\eps}{\varepsilon}
\newcommand{\Sxeps}{\mathcal{S}_{x,\varepsilon}}
\newcommand{\SXeps}[1]{\mathcal{S}_{{#1},\varepsilon}}
\newcommand{\Sxneps}{\SXeps{x_{n}}}
\newcommand{\Sxteps}{\SXeps{x+t}}
\newcommand{\Syeps}[1]{\mathcal{S}_{{#1},\varepsilon}^{(0)}}

\newcommand{\cC}{\mathcal{C}}

\newcommand{\cP}{\mathcal{P}}

\newcommand{\cS}{\mathcal{S}}

\DeclareMathOperator{\QOT}{\mathsf{QOT}}
\DeclareMathOperator{\OT}{\mathsf{OT}}

\DeclareMathOperator{\Int}{int}
\DeclareMathOperator{\id}{Id}

\newcommand{\qforallq}{\quad\mbox{for all}\quad}

\newcommand{\qandq}{\quad\mbox{and}\quad}

\newcommand{\mykill}[1]{}
\usepackage[pdfborder={0 0 0}]{hyperref}
\hypersetup{
  urlcolor = black,
  pdfauthor = {Marcel Nutz},
  pdfkeywords = {Optimal Transport, Quadratic Regularization, Sparsity},
  pdftitle = {Sparsity of Quadratically Regularized Optimal Transport: Scalar Case},
  pdfsubject = {Sparsity of Quadratically Regularized Optimal Transport: Scalar Case},
  pdfpagemode = UseNone
}
\usepackage[capitalize]{cleveref}

\theoremstyle{plain}
\newtheorem{theorem}{Theorem}[section]
\newtheorem{proposition}[theorem]{Proposition}
\newtheorem{lemma}[theorem]{Lemma}
\newtheorem{corollary}[theorem]{Corollary}

\theoremstyle{definition}

\newtheorem{remark}[theorem]{Remark}

\newtheorem{assumption}[theorem]{Assumption}

\theoremstyle{remark}

\renewenvironment{thebibliography}[1]{%
\begin{oldthebibliography}{#1}%
\setlength{\baselineskip}{.9em}
\linespread{1}
\small
\setlength{\parskip}{.40ex}%
\setlength{\itemsep}{.30em}%
}%
{%
\end{oldthebibliography}%
}

\begin{document}

\title{ \vspace{-2em}
Sparsity of Quadratically Regularized Optimal Transport: Scalar Case}
\date{\today}
\author{
  Alberto Gonz{\'a}lez-Sanz%
  \thanks{Department of Statistics, Columbia University, ag4855@columbia.edu} \and 
  Marcel Nutz%
  \thanks{Departments of Mathematics and Statistics, Columbia University, mnutz@columbia.edu. Research supported by NSF Grants DMS-2106056, DMS-2407074.}
  }
  
\maketitle \vspace{-2.5em}

\begin{abstract}
The  quadratically regularized optimal transport problem is empirically known to have sparse solutions: its optimal coupling $\pi_{\varepsilon}$ has sparse support for small regularization parameter $\varepsilon$, in contrast to entropic regularization whose solutions have full support for any $\varepsilon>0$. Focusing on continuous and scalar marginals, we provide the first precise description of this sparsity. Namely, we show that the support of $\pi_{\varepsilon}$ shrinks to the Monge graph at the sharp rate $\varepsilon^{1/3}$. This result is based on a detailed analysis of the dual potential $f_{\varepsilon}$ for small $\varepsilon$. In particular, we prove that $f_{\varepsilon}$ is twice differentiable a.s.\ and bound the second derivative uniformly in $\varepsilon$, showing that $f_{\varepsilon}$ is uniformly strongly convex. Convergence rates for $f_{\varepsilon}$ and its derivative are also obtained. 
\end{abstract}

 \vspace{.5em}

{\small


}
 \vspace{0em}
\section{Introduction}

For probability measures $\mu_0,\mu_1$ on $\R^{d}$, the optimal transport problem with quadratic cost is 
\begin{equation}
    \label{otIntro}
    \OT:=\inf_{\pi\in \Pi(\mu_0, \mu_1)} \int \frac{1}{2}\|x-y\|^2d\pi(x,y)
\end{equation}
where $\Pi(\mu_0, \mu_1)$ denotes the set of couplings (or transport plans) of $\mu_0,\mu_1$. Given a regularization parameter $\eps>0$, we study the quadratically regularized transport problem where couplings are penalized by the squared $L^{2}$-norm of their density,
\begin{equation}\label{qotIntro}
  \QOT_\eps:=  \inf_{\pi\in \Pi(\mu_0, \mu_1)} \int \frac{1}{2}\|x-y\|^2d\pi(x,y) +\frac{\varepsilon}{2}\left\| \frac{d \pi}{ d(\mu_0\otimes \mu_1)}\right\|^2_{L^{2}(\mu_0\otimes \mu_1)}.
\end{equation}
The basic idea of regularized optimal transport is to approximate the linear problem~\eqref{otIntro} by a strictly convex one. Since the influential work of \cite{Cuturi.13} proposing entropic regularization to enable Sinkhorn's algorithm, applications of optimal transport have exploded in machine learning, statistics, image processing, and other domains where distributions or data sets need to be compared (see, e.g., \cite{PeyreCuturi.19}). Regularization has several purposes in this context: facilitating computation as in \cite{Cuturi.13}, improving sample complexity~\cite{genevay2019sample,MenaWeed.19}, obtaining stable couplings~\cite{BayraktarEcksteinZhang.22, GhosalNutzBernton.21b}, etc. Two regularizations are primarily used; entropic regularization penalizes couplings by the Kullback--Leibler divergence while quadratic regularization penalizes by the $L^{2}$-norm of the density. Entropic regularization (EOT) is the most frequent choice and has the strongest smoothness properties (e.g., \cite{Nutz.20, PeyreCuturi.19}).  This smoothness is linked to the full support property: the support of the optimal coupling equals the one of the product $\mu_{0}\otimes\mu_{1}$ for any value of the regularization parameter, even though the unregularized optimal transport is concentrated on a graph. This can be undesirable (``overspreading'') depending on the application---it may correspond to blurrier images in an image processing task as shown in~\cite{blondel18quadratic} or bias in a manifold learning task as in \cite{ZhangMordantMatsumotoSchiebinger.23}. By contrast, quadratic regularization (QOT) is empirically known to give rise to couplings with \emph{sparse support} for small $\eps$, and can therefore be useful in such situations. While EOT has been studied in hundreds of works in the last few years, QOT is traditionally considered less analytically tractable  and theoretical results are scarce. However, recent and ongoing works show that QOT is much more attractive than expected, including parametric sample complexity~\cite{GonzalezSanzDelBarrioNutz.25} and exponentially fast convergence of several algorithms for its numerical solution~\cite{GonzalezSanzNutzRiveros.25,GonzalezSanzNutzRiveros.26}.

Focusing on the scalar case $d=1$ and continuous (cf.\ \cref{as:main}) marginals $\mu_0,\mu_1$, the present work is the first to describe the sparsity of the optimal coupling $\pi_{\eps}$ for~\eqref{qotIntro}. More generally, we provide a fairly complete quantitative picture of QOT for small regularization parameter~$\eps$. Our main result shows that the support $\cS_{\eps}$ of the optimal coupling $\pi_{\eps}$ shrinks at rate $\eps^{\frac13}$. More precisely, this holds uniformly for all sections of the support: there is $C>0$ such that for small $\eps$,
$$
  {C}^{-1}{\varepsilon^{\frac{1}{3}}}\leq |\Sxeps| \leq C\varepsilon^{\frac{1}{3}}
$$
for all $x$, where $\Sxeps$ is the $x$-section of $\cS_{\eps}$ and $|\Sxeps|$ is its diameter; see \cref{th:rate}. Moreover, $\cS_{\eps}$ approaches the support of the limiting optimal transport $\pi_{0}$ at the same rate, in $L^{2}$-norm. Note that the support of $\pi_{0}$ is the graph of the optimal transport map (Monge map) $T_{0}$ inducing $\pi_{0}$. Denoting by $d$ the Hausdorff distance, 
    $$ {C}^{-1}{\varepsilon^{\frac{1}{3}}}\leq \left\|d\left(\Sxeps , T_0(x)\right)\right\|_{L^{2}(dx)} \leq C\varepsilon^{\frac{1}{3}};
    $$
    see \cref{th:ratesDer}.
Our result has a direct interpretation: it guarantees that the solution of the quadratically regularized problem does not place mass far from the optimal transport map, unlike the solution of EOT. We conjecture that the rate generalizes to $d$-dimensional space as $\varepsilon^{\frac{1}{d+2}}$, however at present we can show this only in special cases (such as the self-transport case $\mu_{0}=\mu_{1}$). Our convergence rate confirms the heuristic link between QOT and the porous media equation recently emphasized in \cite{GarrizMolinaGonzalezSanzMordant.24} (see also \cref{se:lit} below): in the small-time limit $t\to0$, the support of the solution to that equation tends to the support of the initial condition at the same rate $t^{\frac13}$ (for $d=1$, and $t^{\frac{1}{d+2}}$ in general). Separately, we establish an interesting geometric fact that has not been documented before (\cref{pr:monotoneSupport}): like the limiting optimal transport $T_{0}(x)$, the support~$\Sxeps$ is increasing in~$x$ (i.e., for $x<x'$ there is a bijection $\phi: \Sxeps\to\SXeps{x'}$ with $\phi(y)\geq y$).

Our results on the support are based on a detailed analysis of the dual potentials $(f_{\eps},g_{\eps})$ associated with $\pi_{\eps}$ via
\begin{equation*}
     \frac{d\pi_\varepsilon}{d(\mu_0\otimes \mu_1)} (x,y)=\left(\frac{ x y-f_\varepsilon(x)-g_\varepsilon(y) }{\varepsilon} \right)_+.
\end{equation*}
While we refer to  \cref{subse:backgroundNotation} for detailed definitions, one may say that $f_{\eps}$ is the analogue of the convex Kantorovich potential $f_{0}$ inducing Monge's map via its derivative $f_{0}'=T_{0}$.  We show that $f_{\eps}$ is $\cC^{1}$ and twice differentiable a.s., and provide expressions for the derivatives. The main technical result, obtained in tandem with sparsity, is uniform strong convexity. More precisely,
$$  C^{-1} \leq f_\varepsilon''(x) \leq C$$
for small $\eps$; see \cref{th:Reg}. We also obtain rates for $f_\varepsilon\to f_{0}$ and $f'_\varepsilon\to f'_{0}=T_{0}$; see \cref{th:ratesDer} and its corollary. 

The proof of our main results is based on a geometric approach that, to the best of our knowledge, is novel. It is easy to bound the following geometric quantity: the product of the diameter $|\Sxeps|$ and the maximum value of the density $\frac{d\pi_\varepsilon}{d(\mu_0\otimes \mu_1)} (x,\cdot)$ . Thus, we study the relation between these two, which is closely related to the second derivative of the potential. Specifically, we bound $f''_{\eps}$ in terms of a ratio of diameters of $x$-sections and $y$-sections of the support. At the same time, we bound these diameters in term of the second derivatives. Analyzing the powers in those relations yields the desired uniform bounds for all quantities. 

We remark that a generalization of our results to the multivariate case would have substantial consequences for the regularity theory of (unregularized) optimal transport because the Kantorovich potential inherits the properties of $f_{\eps}$. Specifically, this approach may enable boundary regularity results along the lines of \cite{ChenLiuWang.21} without the $\alpha$-H\"older conditions on the marginals that seems to be essential for the Monge--Amp\`ere approach. (That may be a hint that the generalization to $d$ dimensions is not straightforward.) We remark that EOT has been used to obtain regularity results in a similar spirit \cite{ChewiPooladian.23}. However, that approach using covariance inequalities is valid only for marginals with full support on $\R^{d}$, hence cannot yield boundary regularity.

While completing this paper, we learned about concurrent and independent ongoing research by Johannes Wiesel and Xingyu Xu who kindly shared their results (now posted in~\cite{WieselXu.24}). Their work has the same goal of bounding the size of the optimal support and distance to the Monge map for small $\eps$, but their approach does not involve studying the shape of the  potentials in detail. Instead, their arguments start with a relatively direct bound on the maximum of the density based on equations~\eqref{eq:solution}, \eqref{eq:solution2}. This bound is very general, allowing for much less restrictive assumptions than in our work (multivariate setting, non-convex marginal supports, unbounded marginal density, etc.) and leading to the first quantitative results in such a setting. On the other hand, it is not accurate enough to deliver the sharp rate.  
Assuming that the Monge map is Lipschitz, the size of the support is bounded from above by $C\eps^{\frac{1}{2d+2}}$ and the distance to the Monge map by $C\eps^{\frac{1}{4(d+1)^{2}}}$. The sharp rate $\eps^{\frac{1}{d+2}}$ is obtained only when the Monge map is the identity.%

\subsection{Related literature}\label{se:lit}

Quadratically regularized optimal transport was first explored by~\cite{blondel18quadratic} in the discrete setting, with experiments highlighting the sparsity and theoretical results including convergence for small regularization parameter. On the other hand, \cite{EssidSolomon.18} studied quadratic regularization for a minimum-cost flow problem on a graph, including discrete optimal transport as a special case. Sparsity is discussed in several examples. Subsequently, \cite{DesseinPapadakisRouas.18} considered optimal transport with convex regularization, again with quadratic regularization as a special case. The asymptotic questions considered in the present work do not have an exact counterpart in the discrete case. Indeed, it is known from~\cite{MangasarianMeyer.79} that the optimal coupling~$\pi_{\eps}$ converges stationarily to the minimum-norm optimal transport $\pi_{0}$; that is, there exists a threshold $\eps_{0}>0$ such that $\pi_{\eps}=\pi_{0}$ for $\eps\leq \eps_{0}$. While one cannot speak of a convergence rate, the recent work ~\cite{GonzalezSanzNutz.24} can be seen as an analogue, providing the value of the threshold $\eps_{0}$. Also in the discrete case, \cite{GonzalezSanzNutzRiveros.24} investigates whether the support of $\pi_{\eps}$ is monotone in~$\eps$ before reaching~$\eps_{0}$, and finds a negative answer.

In the continuous setting, quadratically regularized optimal transport was first explored from a computational point of view. Several works including \cite{EcksteinKupper.21, GeneveyEtAl.16, GulrajaniAhmedArjovskyDumoulinCourville.17, LiGenevayYurochkinSolomon.20, seguy2018large} approach the dual problem by optimization techniques. For instance, \cite{LiGenevayYurochkinSolomon.20} computes regularized Wasserstein barycenters using neural networks and finds that the quadratic penalty produces sharper results than the entropic. More recently, \cite{ZhangMordantMatsumotoSchiebinger.23} uses quadratic regularization in a manifold learning task related to single cell RNA sequencing and notes that sparsity is crucial to avoid biasing the affinity matrix.

The first work rigorously addressing a continuous setting is~\cite{LorenzMannsMeyer.21}, deriving duality results and presenting two algorithms. %
Their problem formulation differs slightly from~\eqref{qotIntro} in that the $L^{2}$-norm is taken  with respect to the Lebesgue measure rather than $\mu_{0}\otimes\mu_{1}$. We prefer the latter as it is generally applicable and  behaves naturally under approximation by sampling or discretization.  More recently, \cite{Nutz.24} derives the basic results of~\cite{LorenzMannsMeyer.21} in a more general setting and studies the uniqueness of the dual potentials. Noting that the potential must converge to the Kantorovich potential for $\eps\to0$ by the Arzel\`a--Ascoli theorem, \cite{Nutz.24} concludes that the support of $\pi_{\eps}$ converges to the one of $\pi_{0}$. This qualitative convergence result is stated in \cref{pr:qualitative} below. 

Prior to the present work, quantitative results were available only for the convergence $\QOT_{\eps} \to \OT$ of the optimal transport \emph{cost}. Indeed, for a more general class of regularizations, a rate is provided in~\cite{EcksteinNutz.22} based on a quantization argument; the qualitative (Gamma) convergence was previously shown in  \cite{LorenzMahler.22}. Specializing the result of~\cite{EcksteinNutz.22}  to quadratic regularization yields $C^{-1} \eps^{\frac{2}{d+2}}\leq \QOT_{\eps} - \OT \leq C \eps^{\frac{2}{d+2}}$. The very recent analysis of \cite{GarrizMolinaGonzalezSanzMordant.24} shows that this power is the exact leading term and identifies the constant:
\begin{equation*}
    \lim_{\varepsilon\to 0}  \frac{\QOT_\eps -\OT }{\varepsilon^{\frac{2}{d+2}}}  = C_d  \int_{\R^{d}} \big(u_0(x)u_1(T_0(x))\big)^{-\frac{1}{(d+2)}}d \mu_0(x)
\end{equation*}
where $C_{d}$ is an explicit dimensional constant and $u_{i}$ is the density of $\mu_{i}$. This result is obtained by establishing a link to the porous media equation $\partial_t u =\Delta u^2 $. Namely, an approximate solution of QOT is obtained by modifying the Barenblatt profile which is a fundamental solution of the porous media equation (see \cite{Vazquez.07}). Based on the shape of the approximate solution in \cite{GarrizMolinaGonzalezSanzMordant.24} one can conjecture that support of the true solution also shrinks at rate $\varepsilon^{\frac{1}{d+2}}$, which is confirmed by our results.

Finally, we mention that the limit $\eps\to0$ has been studied in detail for the EOT problem, starting with \cite{CominettiSanMartin.94} for discrete and~\cite{Mikami.02, Mikami.04} (see also \cite{ChenGeorgiouPavon.16, Leonard.12}) for continuous marginals. See \cite{AltschulerNilesWeedStromme.21,Berman.20,BerntonGhosalNutz.21, CarlierDuvalPeyreSchmitzer.17, ConfortiTamanini.19, NutzWiesel.21, Pal.19}, among others. Of course, the optimal support has not been of much interest in EOT, as it is simply the support of $\mu_{0}\otimes\mu_{1}$. There is a direct analogy between the role of the heat equation in \cite{Pal.19} for EOT and the porous media equation in \cite{GarrizMolinaGonzalezSanzMordant.24} for QOT. In the present work, our arguments follow a novel route. 

\vspace{.8em}

The remainder of this paper is organized as follows. \Cref{se:main} details notation and background, then states the main results.  \Cref{sec:convex} establishes convexity and determines the first derivative of the potential $f_{\eps}$, whereas \cref{se:SecondOrderDer} is devoted to the second derivative. The main part is \cref{se:Bounds2ndDerivative}, bounding the second derivative and the diameter of the support. Finally, \cref{section:ratesPotentials} details how those results imply the  convergence rate to unregularized optimal transport.

\section{Main Results}\label{se:main}

\subsection{Background and Notation}\label{subse:backgroundNotation}

Let $\mu_0, \mu_1\in \mathcal{P}(\R)$ be probability measures on $\R$ with finite second moments and $\Pi(\mu_0, \mu_1)$ their set of couplings; i.e., $\pi\in \mathcal{P}(\R\times \R)$ such that $\pi(\cdot\times \R)=\mu_0(\cdot)$ and $\pi(\R\times \cdot)=\mu_1(\cdot)$. Given $\varepsilon>0$, the quadratically regularized optimal transport problem between $\mu_0$ and $\mu_1$ is
\begin{equation}
    \label{qot}
  \inf_{\pi\in \Pi(\mu_0, \mu_1)} \int \frac{1}{2}(x-y)^2d\pi(x,y) +\frac{\varepsilon}{2} \int \left( \frac{d \pi}{ d(\mu_0\otimes \mu_1)} (x,y)\right)^2 d(\mu_0\otimes \mu_1)(x,y).
\end{equation}
We refer to \cite{Nutz.24} for the following facts. The problem~\eqref{qot} has a unique minimizer, denoted $\pi_{\varepsilon}\in\Pi(\mu_0, \mu_1)$. The dual problem of~\eqref{qot} is
\begin{multline}
    \label{dual}
    \sup_{(f,g)\in L^2(\mu_0)\times L^2(\mu_1)} \int f(x) d\mu_0(x)+\int g(y) d\mu_1(y)\\-\frac{1}{2 \varepsilon }\int{  \left({f(x)+g(y)-\frac{(x-y)^2}{2}} \right)_+^2 }d(\mu_0\otimes \mu_1)(x,y).
\end{multline}
Given any solution $(\tilde f,\tilde g)$ of \eqref{dual}, we have the ``optimality condition''
\begin{equation}
    \label{qdensity1}
    \frac{d\pi_\varepsilon}{d(\mu_0\otimes \mu_1)} (x,y)=\left(\frac{\tilde f(x)+ \tilde g(y) -\frac{1}{2}(x-y)^2}{\varepsilon} \right)_+.
\end{equation}
To be in line with convex analysis, it will be convenient to reparametrize the solution as 
\begin{align}\label{ex:reparam}
  \tilde f(x) = \tfrac12 x^2 -f_\varepsilon(x), \qquad \tilde g(x) = \tfrac12y^2 -g_\varepsilon(x),
\end{align} 
so that~\eqref{qdensity1} becomes
\begin{equation}
    \label{qdensity}
    \frac{d\pi_\varepsilon}{d(\mu_0\otimes \mu_1)} (x,y)=\left(\frac{ x y-f_\varepsilon(x)-g_\varepsilon(y) }{\varepsilon} \right)_+.
\end{equation}
Under \cref{as:main} below, the ``potentials'' $(f_\varepsilon, g_\varepsilon)$ are unique (a.s.) up to a parallel shift; that is, up to changing $(f_\varepsilon, g_\varepsilon)$ to $(f_\varepsilon + \alpha, g_\varepsilon-\alpha)$ for some $\alpha\in\R$. Moreover, we may choose $(f_\varepsilon, g_\varepsilon)$ to be (everywhere defined and) Lipschitz continuous on~$\R$. See \cite[Theorem 3.7, Lemma~2.5]{Nutz.24}.

Next, we introduce standard notation and results for the (unregularized) optimal transport problem (see \cite{Villani.09}),
\begin{equation}
    \label{ot}
    \inf_{\pi\in \Pi(\mu_0, \mu_1)} \int \frac{1}{2}(x-y)^2d\pi(x,y).
\end{equation}
For absolutely continuous marginal $\mu_{0}$, this problem has a unique solution $\pi_{0}\in \Pi(\mu_0, \mu_1)$. Moreover, $\pi_{0}=(\id\times T_0)_\# \mu_0$ where $T_{0}:\R\to\R$ is the optimal transport map, often called Monge map. Here $\#$ denotes push-forward, $f_{\#}(\mu) := \mu(f^{-1}(\cdot))$. In the present scalar case, it holds that $T_0=F_1^{-1}\circ F_0$, where $F_i(x)=\mu_i((-\infty, x])$ is the cumulative distribution function of $\mu_i$. On the other hand, $T_0= f'_{0}$ is the derivative of a convex function $f_{0}$ called the Kantorovich potential. Namely, $(f_0,g_0)$ solves  the dual problem
\begin{equation}
    \label{dualSemi}
    \inf_{(f,g): \ f(x)+g(x)\geq xy} \int f (x) d\mu_0(x)+\int g(y) d\mu_1(y).
\end{equation}
Here we have applied the same reparametrization as in~\eqref{ex:reparam} to work directly with convex functions---indeed, $(f_0,g_0)$ is the limit of $(f_\varepsilon, g_\varepsilon)$ from~\eqref{qdensity}. Again, $(f_0,g_0)$ is unique up to parallel shift under \cref{as:main} below.

Some more notation will be useful. For brevity,
\begin{align*}
  \mbox{$|B|$ denotes the Lebesgue measure of a Borel set $B\subset\R$.}
\end{align*} 
We will use this primarily when $B$ is an interval, so that $|B|$ is also the length.
As usual,  the essential supremum of $f:\Omega_0\to \R$  is 
$
\esssup_{\Omega_0} f = \inf \{ M \in \R: |\{ x \in \Omega_0 : f(x) > M \}| = 0 \},
$
and similarly $\essinf_{\Omega_0} f = -\esssup_{\Omega_0} (-f)$. Finally, $\chi_{A}$ denotes the (measure-theoretic) indicator function, $\chi_{A}(x)=1$ for $x\in A$ and $\chi_{A}(x)=0$ for $x\notin A$.

\subsection{Results}

All results are stated under the following standard condition on the marginals.

\begin{assumption}\label{as:main}
  The marginals $\mu_{i}\in\cP(\R)$ are compactly supported and have continuous densities bounded from above and below. That is,  for $i=0,1$, there are compact intervals $\Omega_i =[a_i, b_i]\subset \R$ and $u_i\in \mathcal{C}(\Omega_i)$ such that $d\mu_i= u_i dx\in  \mathcal{P}(\Omega_i)$ and 
\begin{align}\label{eq:asmainBounds}
0< \lambda \leq  u_i \leq \Lambda <\infty\quad \text{on}\ \Omega_i
\end{align}
for some constants $\lambda,\Lambda$.
\end{assumption} 

The results below are stated for the potentials $f_{\eps}$ and the $x$-sections of the support. As \cref{as:main} is symmetric, the same results hold for $g_{\eps}$ and the $y$-sections. 
Our first result establishes the regularity of the potential $f_{\eps}$ and most importantly its \emph{strong convexity} uniformly in $\eps$.  

\begin{theorem}[Regularity]\label{th:Reg}
    There exist $C,\varepsilon_0>0$ such that for every $0<\varepsilon<\varepsilon_0$,
\begin{enumerate}
    \item  $f_\varepsilon$ is convex and in $\mathcal{C}^{1}(\Omega_0)$ (see \cref{ContinuityofT} for a formula for $f'_{\eps}$);
        \item $f_\varepsilon$ is twice differentiable a.e.\footnote{In fact, twice differentiable except at two particular points $x^{(m)}, x^{(M)}$; see \cref{se:SecondOrderDer}.} (see \cref{pr:fDerivativeFormula} for a formula for $f''_{\eps}$);
    \item $f_\varepsilon$ is uniformly strongly convex with\footnote{Similarly, one can take the pointwise $\sup$ and $\inf$ over $\Omega_0\setminus \{x^{(m)}, x^{(M)}\}$.}
    $$ \esssup_{\Omega_0} f_\varepsilon'' \leq C \quad \text{and} 
    \quad  \essinf_{\Omega_0} f_\varepsilon'' \geq C^{-1}.$$
\end{enumerate}
The constant~$C$ depends only on the constants and the moduli of continuity of $u_{i}$ in \cref{as:main}. 
\end{theorem}

The condition that $\eps$ be small is essential for strong convexity. In fact, $f_\varepsilon$ is affine for large $\eps$; see \cref{rk:2ndDerivative}.

\begin{remark}[Necessity of \cref{as:main}]
Given the uniform convergence $f_{\eps}\to f_{0}$ (see \cite{Nutz.24} or below), the assertion of \cref{th:Reg} implies that the transport map $T_{0}=f'_{0}$ is Lipschitz with (a.s.\ defined) derivative bounded away from $0$ and infinity. Recall that $T_0=F_1^{-1}\circ F_0$ in the present scalar case, whence  the marginals $\mu_{i}$ have densities bounded above and below. In that sense, \eqref{eq:asmainBounds} is necessary for \cref{th:Reg} to hold. (The continuity of the densities could potentially be relaxed.)
\end{remark}

\begin{remark}[Multivariate case]
	We can extend \cref{th:Reg}\, (i),(ii) to the multivariate case. However, the key result (iii) is much more delicate, which is why we leave the multivariate case for future work. We observe that (iii) implies in particular a similar regularity result for the Kantorovich potential. It is known that the (boundary) regularity of the Kantorovich potential is a very deep result in the multivariate case  \cite{Caffarelli.96,Urbas.97,ChenLiuWang.21}.
\end{remark}

Our second and main result quantifies the sparsity of the optimal coupling $\pi_{\eps}$. Specifically, we describe the diameter of the sections of the support,
$$ \Sxeps :=\{ y \in \Omega_1: \ xy - f_\varepsilon( x ) - g_\varepsilon( y ) \geq 0 \}.$$
Note that each section $\Sxeps$ is an interval due to the convexity of $g_\varepsilon$, hence the diameter is equal to the Lebesgue measure $|\Sxeps|$. Our result shows that the diameter is (exactly) of order~$\varepsilon^{\frac{1}{3}}$, uniformly in~$x$.

\begin{theorem}[Sparsity]\label{th:rate}
    There exist $C,\varepsilon_0>0$ such that 
    $$ {C}^{-1}{\varepsilon^{\frac{1}{3}}}\leq |\Sxeps| \leq C\varepsilon^{\frac{1}{3}} $$
    for all $x\in \Omega_0$ and $\varepsilon<
    \varepsilon_0$. 
\end{theorem}

While \cref{th:rate} gives the sharp rate at which the  support of $\pi_{\eps}$ shrinks, this alone does not imply that the support is close to the unregularized optimal transport. The next result provides the sharp convergence rate for the Hausdorff distance 
$$ d\left(\Sxeps , T_0(x)\right)=\sup_{y\in \Sxeps}|y -T_0(x)| $$
between $\Sxeps$ and the image $\{T_0(x)\}$ of Monge's map; note that the latter is the $x$-section of the support of $\pi_{0}$. The convergence is  measured in $L^2(\Omega_0)$-norm 
$\|h\|_{ L^2(\Omega_0)} := \int_{\Omega_0} h(x)^2 dx $. As a consequence, we also obtain the rate  of convergence of $f'_\varepsilon $ towards $T_0$ in terms of the same norm.

\begin{theorem}[Convergence]\label{th:ratesDer}
    There exist $C,\varepsilon_0>0$ such that  for every $\varepsilon<\varepsilon_0$,
    \begin{enumerate}
        \item  the Hausdorff distance between the optimal supports satisfies
    $$ {C}^{-1}{\varepsilon^{\frac{1}{3}}}\leq \left\|d\left(\Sxeps , T_0(x)\right)\right\|_{L^{2}(\Omega_{0},dx)} \leq C\varepsilon^{\frac{1}{3}}, $$
    \item the derivatives $f'_\varepsilon$  converge to Monge's map $T_0=f'_0$ according to
    \begin{equation}
        \label{eq:rateDerivative}
        \|f'_\varepsilon-T_0\|_{L^2(\Omega_0,dx)} \leq C\varepsilon^{\frac{1}{3}}.
    \end{equation}
    \end{enumerate}
\end{theorem}

The rate~\eqref{eq:rateDerivative} for the derivatives of the potentials implies a rate for the potentials themselves. As the potentials are only determined up to an additive constant, convergence can only hold if the constant is chosen in a suitable way. For instance, we can normalize the potentials symmetrically; i.e.,
\begin{align}\label{eq:normalizationPot}
   \int f_{\eps}(x) \mu_{0}(x) = \int g_{\eps}(y)\mu_{1}(y), \qquad \int f_{0}(x) \mu_{0}(x) = \int g_{0}(y)\mu_{1}(y).
\end{align} 
Another popular normalization is $\int g_{\eps}(y)\mu_{1}(y)= \int g_{0}(y)\mu_{1}(y)=0$;  the subsequent result also holds for that choice. We use the H\"older norm 
\[
\|h\|_{C^{0, \frac{1}{2}}(\Omega_0)} := \sup_{x \in \Omega_0} |h(x)|+  \sup_{x, y \in \Omega_0, x \neq y} \frac{|h(x) - h(y)|}{|x - y|^{\frac{1}{2}}}.
\]

\begin{corollary}\label{co:potentialRate}
    Let the potentials be normalized according to~\eqref{eq:normalizationPot}.
    There exist $C,\varepsilon_0>0$ such that  for every $\varepsilon<\varepsilon_0$,
    $$
       \|f_\varepsilon-f_0\|_{C^{0, \frac{1}{2}}(\Omega_0)} \leq C\varepsilon^{\frac{1}{3}}.
    $$
\end{corollary}

\begin{remark}\label{rk:sharpness}
  We conjecture that the rate~\eqref{eq:rateDerivative} is sharp, whereas the rate for $\|f_\varepsilon-f_0\|$ in \cref{co:potentialRate} is not. The sharp rate for $\|f_\varepsilon-f_0\|$ is conjectured to be $\varepsilon^{\frac{2}{3}}$.
\end{remark} 

\begin{remark}
In the preceding results, the asymptotic regime is asserted for regularization parameters $\eps<\eps_0$. Alternately, we may assert the asymptotic regime when the size of the support is below a threshold. Specifically, let $ \omega_i$ denote the modulus of continuity of $\log(u_i)$; then the purpose of $\eps_0$ is to ensure that for $\eps\leq \eps_0$, 
$$  \sup_{x\in \Omega_0}|\Sxeps|\leq \frac{1}{3}|\Omega_1|, \qquad \sup_{x\in \Omega_0} \omega_1(|\Sxeps|)\leq  \log(3/2),$$ 
and symmetrically with reversed roles of $x$ and $y$. The first condition is used in \cref{se:SecondOrderDer,subse:Bounds2ndDerivative}, the second in \cref{lemma:distanceBound}.
\end{remark}

\subsection{Outline of the proofs}

\cref{th:Reg,th:rate} are proved simultaneously though several steps. First, in \cref{sec:convex}, we prove that $f_\varepsilon$ is differentiable with a monotone increasing derivative $f_\varepsilon'(x)=\frac{\int_{\Sxeps} y   d\mu_1( y )}{\mu_1 (\Sxeps )}$, which implies that $f_\varepsilon$ is convex. In \cref{se:SecondOrderDer}, we show that $f_\varepsilon$ is twice differentiable (except possibly at two points) and provide an explicit formula for $f_\varepsilon''(x)$. 
In \cref{se:Bounds2ndDerivative}, this formula is used to bound $f_\varepsilon''(x)$. First, in \cref{Sec:Comparison}, we provide a relationship between the diameter $|\Sxeps|$ of the support and the maximum (over $y$) of the density of $\pi_\varepsilon$, whereas \cref{subse:dist} shows that $f'_\varepsilon(x)$ is approximately at the center of the support $\Sxeps$. In \cref{subse:Bounds2ndDerivative}, we give an initial bound on $f_\varepsilon''(x)$ in terms of the ratio between the diameters of different supports. Roughly, the bound is of the form 
  $$ C^{-1}\inf_{y\in \partial \Sxeps }\frac{ |\Sxeps|}{|\Syeps{y}|}\leq  f_\varepsilon'' ( x )\leq  C  \sup_{y\in \partial \Sxeps}\frac{ |\Sxeps|}{|\Syeps{y}|},$$
where $\Syeps{y}$ is the $y$-section of the support (the symmetric analogue of $\Sxeps$ with reversed roles for $x,y$). We proceed by relating and simultaneously bounding the four objects $\Sxeps$, $f''_{\eps}$ and $\Syeps{y}, g''_{\eps}$. A key intermediate result is a bound for $|\Sxeps|$ in terms of $\sigma_M(f_\varepsilon)= \esssup_{\Omega_0} f_\varepsilon'' $ and $\sigma_m(f_\varepsilon)= \essinf_{\Omega_0} f_\varepsilon'' $, namely 
   \begin{equation}
       \label{RelationSupportSigmaIntro}
       \left(\frac{\varepsilon}{\sigma_M(f_\varepsilon)}\right)^{\frac{1}{3}} \leq  |\Sxeps| \leq \left(\frac{\varepsilon}{\sigma_m(f_\varepsilon)}\right)^{\frac{1}{3}},
   \end{equation}
   as well as an analogous bound for $|\Syeps{y}|$ in terms of $\sigma_M(g_\varepsilon)= \esssup_{\Omega_1} g_\varepsilon'' $ and $\sigma_m(g_\varepsilon)= \essinf_{\Omega_1} g_\varepsilon''$. Using these bounds, we are able to relate the derivatives of $f_\varepsilon$ and $g_\varepsilon$ by
\begin{alignat*}{2} 
\sigma_M(f_\varepsilon) &\leq  C \left(\frac{\sigma_M(g_\varepsilon)}{\sigma_m(f_\varepsilon)}\right)^{\frac{1}{3}},  \qquad \sigma_m(f_\varepsilon) &&\geq  \frac{1}{C} \left(\frac{\sigma_m(g_\varepsilon)}{\sigma_M(f_\varepsilon)}\right)^{\frac{1}{3}}, \\
\sigma_M(g_\varepsilon)&\leq  C \left(\frac{\sigma_M(f_\varepsilon)}{\sigma_m(g_\varepsilon)}\right)^{\frac{1}{3}},  \qquad  \sigma_m(g_\varepsilon)&&\geq  \frac{1}{C} \left(\frac{\sigma_m(f_\varepsilon)}{\sigma_M(g_\varepsilon)}\right)^{\frac{1}{3}}  .
\end{alignat*}
After some algebra, this implies uniform bounds for all four quantities $\sigma_m(f_\varepsilon), \sigma_M(f_\varepsilon)$, $\sigma_m(g_\varepsilon)$, $\sigma_M(g_\varepsilon)$, establishing the results on $f_\varepsilon''(x)$ in \cref{th:Reg}. Now, \eqref{RelationSupportSigmaIntro} also yields the rate of $|\Sxeps|$ in \cref{th:rate}.

The proof of \cref{th:ratesDer} is obtained from the rate of $|\Sxeps|$, the rate for the regularized transport cost from \cite{EcksteinNutz.22} and an argument that we adapt from \cite{CarlierPegonTamanini.22}.

\section{Convexity and first order regularity}\label{sec:convex}

In this section we prove that the potential $f_\varepsilon$ is a $\mathcal{C}^1$, convex function. Recall the notation introduced in \cref{as:main} (which is in force for the remainder of the paper) and the continuous functions $f_\varepsilon$ and  $g_\varepsilon$  defined by \eqref{qdensity}, which satisfy (see \cite{Nutz.24}) 
 \begin{align}
        \label{eq:solution}
     &\int (  x  y -f_\varepsilon( x )-g_\varepsilon( y ))_+ d\mu_1( y )= \varepsilon \quad \text{for all $ x \in 
  \Omega_0$},\\
     \label{eq:solution2}
     &\int (  x  y -f_\varepsilon( x )-g_\varepsilon( y ))_+ d\mu_0( x )= \varepsilon \quad \text{for all $ y \in 
 \Omega_1$}.
\end{align}
These equations express the fact that the measure $\pi_{\eps}$ of~\eqref{qdensity} has marginals $\mu_{0}$ and $\mu_{1}$. Recall also the sections of the optimal support,
$$ \Sxeps =\{ y \in \Omega_1: \ xy -f_\varepsilon( x )-g_\varepsilon( y ) \geq 0 \},$$
and note that \eqref{eq:solution} implies $\mu_1(\Sxeps)>0$ for any $x \in \Omega_0$. 

\begin{lemma}\label{lemma:BoundOnderivatives}
  Let $ x \in \Omega_0$ and $t>0$ be such that $ x +t\in \Omega_0$. Then
     $$  T_\varepsilon( x ):= \frac{\int_{\Sxeps} y   d\mu_1( y )}{\mu_1 (\Sxeps)}  \quad\mbox{satisfies}\quad T_\varepsilon( x +t) \geq  \frac{f_\varepsilon( x +t)-f_\varepsilon( x )}{t} \geq  T_\varepsilon( x ).$$
    In particular, %
  $T_\varepsilon$ is monotone.
\end{lemma}
\begin{proof}
The inequality 
\begin{equation}
    \label{ineq:abPositive}
    (a)_+-(b)_+\leq \chi_{a\geq 0}(a-b) , \quad a,b\in \R,
\end{equation}
where $\chi_{a\geq 0}:=1$ if $a\geq 0$ and $\chi_{a\geq 0}:=0$ otherwise,
yields
\begin{equation}
    \label{ineq:1}
    (  x   y -f_\varepsilon( x )- g_\varepsilon( y ) )_+ \leq (  (x +t)  y -f_\varepsilon( x +t )- g_\varepsilon( y ) )_+  +\chi_{\Sxeps}\left(f_\varepsilon( x +t )- f_\varepsilon( x )-t    y \right)
\end{equation}
as well as 
\begin{multline}\label{ineq:2}
    ( x   y -f_\varepsilon( x )- g_\varepsilon( y ) )_+\geq  (   (x +t)y -f_\varepsilon( x +t)- g_\varepsilon( y ) )_+ \\ + \chi_{\Sxteps}\left(f_\varepsilon( x +t)- f_\varepsilon( x )-ty \right).
\end{multline}
Integrating \eqref{ineq:1} with respect to $\mu_1$ and using~\eqref{eq:solution} yields 
     $ f_\varepsilon( x +t)-f_\varepsilon( x ) \geq  tT_\varepsilon( x ).$
Similarly, integrating \eqref{ineq:2} yields the second inequality in the claim.
\end{proof}

Noting that $T_\varepsilon(x)\in\Omega_{1}$ is uniformly bounded, \cref{lemma:BoundOnderivatives} implies in particular that $f_{\varepsilon}$ is Lipschitz. Next, we prove that in fact $f_\varepsilon\in \mathcal{C}^1(\Omega_0)$. We write $\Int A$ for the interior of a set~$A$.

\begin{proposition}\label{ContinuityofT}
  For all $\eps>0$, the function $f_\varepsilon$ is convex on $\Omega_0$. Moreover, $f_\varepsilon\in \mathcal{C}^1(\Omega_0)$ with derivative 
\begin{align}\label{eq:fprime}
f_\varepsilon'(x)=\frac{\int_{\Sxeps} y   d\mu_1( y )}{\mu_1 (\Sxeps )}\in \Int\Sxeps \qforallq x\in\Omega_{0}.
\end{align} 
\end{proposition}
\begin{proof}
As $T_\varepsilon$ is monotone by \cref{lemma:BoundOnderivatives}, there is a l.s.c.\ convex function $\phi:\Omega_0\to \R$ such that  
 $$T_\varepsilon( x )\in \partial \phi( x ):=\{ y \in \mathbb{R}: \phi(z) \geq \phi(x) +  y(z - x) \ \text{for all} \, z \in \mathbb{R} \}
$$ for all $ x \in \Omega_0$; cf.\ \cite[Theorem~12.17]{RockafellarWets.98}. Since $\partial \phi( x )$ is the singleton $\{\phi'( x )\}$ for a.e.\ $ x \in \Omega_0$, it follows that $T_\varepsilon( x )=\phi'( x )$  for a.e.\ $ x \in \Omega_0$. Call $\Omega_0'$ the set of all $ x \in \Omega_0$ where the latter holds. Then  $T_\varepsilon$ is continuous on $\Omega_0'$; cf.\ \cite[Theorem~25.5.]{Rockafellar.97}. On the other hand, since $f_\varepsilon$ is  Lipschitz, Rademacher's theorem yields a set $\Omega''_0\subset \Omega_0$ with  $|\Omega_0\setminus \Omega''_0|=0$ on which $f_\varepsilon$ is differentiable. In view of \cref{lemma:BoundOnderivatives} we conclude that 
    $ f'_\varepsilon= T_\varepsilon=\phi'$ on $\Omega'_0\cap \Omega''_0$. This implies that $f_\varepsilon=\phi+\alpha$ on $\Omega_0$, for some constant $\alpha$ (cf.\ \cite[Lemma~2.5]{delBarrioGonzalezLoubes.24}), and completes the proof that $f_{\eps}$ is convex. Symmetrically, $g_{\eps}$ is convex, which in particular implies that $\Sxeps$ is convex; i.e., an interval. (The convexity of the support $\Omega_{1}$ is crucial here.)

	Let $\{ x _n\} \subset \Omega_0$ satisfy $ x _n\to  x \in \Omega_0$. As $f_\varepsilon$ is continuous, the boundary points of the (nonempty) intervals $\Sxneps$ converge to the ones of $\Sxeps$. Since $\mu$ is atomless, this implies 
\begin{equation*}
    \mu_1(\Sxneps) \to {\mu_1(\Sxeps)} \qandq \int_{\Sxneps} y   d\mu_1( y ) \to \int_{\Sxeps} y   d\mu_1( y ).
\end{equation*}
As a consequence, $T_{\eps}$ is continuous, which by \cref{lemma:BoundOnderivatives} implies that $f_\varepsilon\in \mathcal{C}^1(\Omega_0)$ with $f'_{\eps}(x)=T_\varepsilon( x )=\frac{\int_{\Sxeps} y   d\mu_1( y )}{\mu_1 (\Sxeps )}$. This weighted average over $\Sxeps$ is an element of $\Sxeps$ by convexity, and it cannot lie on the boundary as $\mu_1$ is absolutely continuous.
\end{proof}

\section{Second order differentiability}\label{se:SecondOrderDer}

This section investigates the sign and the differentiability properties of the first derivative $f'_{\eps}$. 
For brevity, define 
$$ \xi(x,y):=-f_\varepsilon(x)-g_\varepsilon(y)+xy.$$
For each $x$, the function $y\mapsto \xi(x,y)$ is concave, so that the set 
$ \Sxeps=\{y\in\Omega_{1}: \ \xi(x,y)\geq0\}$ is a compact interval
    $$\Sxeps =[y_{m}(x), y_M(x)].$$ 
(The dependence on $\eps$ is suppressed in some of our notation.) For ease of reference, we state the qualitative result of~\cite[Theorem~4.1]{Nutz.24} as follows.

\begin{proposition}\label{pr:qualitative}
  The support of $\pi_{\eps}$ converges to the support of $\pi_{0}$ in Hausdorff distance $d$ as $\eps\to0$. Since the support of $\pi_{0}$ is the graph of the Lipschitz map $T_{0}$, it follows that 
$$\lim_{\eps\to0}\sup_{x\in\Omega_{0}} d(\Sxeps, T_{0}(x))=0 \qandq \lim_{\eps\to0}\sup_{x\in\Omega_{0}} |\Sxeps|=0.$$
\end{proposition}

We already know that $$f_\varepsilon' ( x )= \frac{\int_{\Sxeps} y   u_1( y ) dy}{\mu_1(\Sxeps)}  $$
is monotone, hence differentiable a.e. We shall derive an expression for the derivative $f_\varepsilon'' ( x )$.  Note that numerator and denominator are both of the form 
\begin{align}\label{eq:Xi}
  \Xi( x )=  \int_{y_{m}(x)}^{y_M(x)} h(y) d y
\end{align}
    for a continuous function $h$ on $\Omega_1$. We first establish that $y_{m}(x)$ and $y_{M}(x)$ are continuous functions of~$x$ which are differentiable except at those points~$x$ where the boundary of $\Sxeps$ overlaps with the boundary of $\Omega_{1}$. The continuity follows from the continuity of $f_{\eps}$ and convexity of $g_{\eps}$ (recall that $|\Sxeps|>0$ for all $x\in\Omega_{0}$); cf.\ \cref{ContinuityofT}. Turning to the differentiability, recall from \cref{as:main} that $\Omega_1 =[a_1, b_1]$. We start by 
claiming that
\begin{align*}
   \{x\in \Omega_0: \ y_{M}(x)=b_1  \ {\rm and}\ \xi(x,  y_{M}(x))=0\} \quad&\mbox{is at most a singleton, denoted $\{x^{(M)}\}$,}\\
   \{x\in \Omega_0: \ y_{m}(x)=a_1  \ {\rm and}\ \xi(x,  y_{m}(x))=0\} \quad &\mbox{is at most a singleton, denoted $\{x^{(m)}\}$. }
\end{align*} 
By~\eqref{eq:fprime}, for any $(x,y)\in\Omega_{0}\times\Omega_{1}$,
\begin{align}\label{eq:forContrad}
   y=f'_{\eps}(x) \quad \mbox{implies} \quad y\in (y_{m}(x),y_{M}(x)) \subset(a_{1},b_{1}).
\end{align}     
Fix $y\in \Omega_1$ and suppose there are two points $x_1< x_2$ in $\Omega_{0}$ such that 
$ 0= \xi(x_1,y)= \xi(x_{2},y).$
Then as $ x\mapsto \xi(x,y)=x y-g_\varepsilon(y)-f_\varepsilon(x)$ is differentiable, there exists $x\in(x_{1},x_{2})$ with 
$$
  0 = \partial_{x}\xi(x,y) = y-f'_\varepsilon(x).
$$    
Hence~\eqref{eq:forContrad} shows $y\neq a_{1},b_{1}$ and both claims follow.

The qualitative convergence result \cref{pr:qualitative} shows  $\sup_{x\in \Omega_0}|y_{M}(x)-y_{m}(x)|\to0$ as $\eps\to0$. In particular, we can pick $\varepsilon_0>0$ such that $\sup_{x\in \Omega_0}|y_{M}(x)-y_{m}(x)|\leq \frac{1}{3}|\Omega_1|$ for $\varepsilon<\varepsilon_0$, which implies that $\Sxeps$ can include \emph{at most one} of the two boundary points of $\Omega_1$. Fix $\varepsilon<\varepsilon_0$. We partition $\Omega_0\setminus \{x^{(m)}, x^{(M)}\}$ into the three (relatively open) sets where $\Sxeps$ includes the lower, no, or upper boundary point:
	\begin{align*}
	\Omega_0^{(1)}&=\{x\in \Omega_0: \ y_{M}(x)<b_1 \ {\rm and}\ \xi(x,  y_{m}(x))>0\},\\ 
  \Omega_0^{(2)}&=\{x\in \Omega_0: \ a_1<y_{m}(x)<y_{M}(x)<b_1\},\\
  \Omega_0^{(3)}&=\{x\in \Omega_0: \ \xi(x,  y_{M}(x))>0 \ {\rm and}\ y_{m}(x)>a_1  \}.    
\end{align*}
If $x\in \Omega_0^{(2)}\cup \Omega_0^{(3)}$, then $y_{m}(x)$ is a zero of 
$$ y\mapsto  \xi(x,y)= xy-g_\varepsilon(y)-f_\varepsilon(x). $$
Since
$ \partial_y \xi(x,y)= x- g_\varepsilon'(y) $
cannot be zero at $y_{m}(x)$ by~\eqref{eq:forContrad}, the implicit function theorem yields 
\begin{align}\label{eq:ymderiv}
y_{m}'(x) =-\frac{y_m(x)- f_\varepsilon'(x)}{x- g_\varepsilon'(y_{m}(x))}. 
\end{align} 
(For $x= x^{(m)}$, this holds only if $y_{m}'(x)$ is interpreted as the right derivative, whereas the left derivative is zero.)  
Note that $y_m(x)- f_\varepsilon'(x)<0$ by \eqref{eq:forContrad} whereas $x- g_\varepsilon'(y_{m}(x))=\partial_y \xi(x,y_{m}(x)) >0$ by concavity. Hence $y_{m}'(x)>0$, showing that $y_{m}$ is strictly increasing on $\Omega_0^{(2)}\cup \Omega_0^{(3)}$.
For $x\in \Omega_0^{(1)}\cup \Omega_0^{(2)}$, the same argument yields 
\begin{align}\label{eq:yMderiv}
y_{M}'(x) =-\frac{y_M(x)- f_\varepsilon'(x)}{x- g_\varepsilon'(y_{M}(x))}>0
\end{align}
and strict monotonicity of  $x\mapsto y_{M}'(x)$ on $\Omega_0^{(1)}\cup \Omega_0^{(2)}$. In particular, we have established that $y_{m},y_{M}$ are increasing continuous functions. The definition of the sets $\Omega_0^{(i)}$ then implies that each $\Omega_0^{(i)}$ is an interval.  

Next, we check that $\Omega_0^{(i)}\neq\emptyset$ for all $i$. This is clear for $\Omega_0^{(2)}$ (otherwise, by the choice of~$\varepsilon_0$, no mass is transported to the middle third of $\Omega_{1}$). 
We claim that $g_\varepsilon' (a_{1})\in\Omega_0^{(1)}$. Note first that, by the symmetric counterpart of~\eqref{eq:fprime}, $$g_\varepsilon' (a_{1})=\frac{\int_{\Syeps{a_{1}}} x d\mu_0(x)}{\mu_0(\Syeps{a_{1}})}\in \Syeps{a_{1}}\subset\Omega_{0},$$  
where $\Syeps{y}:=\{ x \in \Omega_0: \ f_\varepsilon( x )+g_\varepsilon( y )-x y \leq 0 \}$. 
The concave function $\xi(g_\varepsilon' (a_{1}),\cdot)$ attains its maximum at any $y\in \Omega_1$ solving  $ g_\varepsilon' (a_{1})-g'_{\eps}(y) =0$. Clearly $y=a_{1}$ satisfies the latter, so that $a_{1}\in \argmax_{y\in \Omega_1}\xi(g_\varepsilon' (a_{1}),y)$. 
Hence, we must have $\xi(g_\varepsilon' (a_{1}),a_1)>0$, for otherwise $\mu_{1}(\Syeps{a_{1}})=0$ contradicting~\eqref{eq:solution}. It follows that $y_m(g_\varepsilon' (a_{1}))= a_1$ and, as $|\SXeps{g_\varepsilon' (a_{1})}|\leq \frac13 |\Omega_{1}|$, that $y_M(g_\varepsilon' (a_{1}))< b_1$. In summary, $g_\varepsilon' (a_{1})\in\Omega_0^{(1)}$. 
Similarly, $g_\varepsilon'(b_{1})\in\Omega_0^{(3)}$. 

As $\Omega_0^{(i)}\neq\emptyset$ for all $i$, it follows that the points $x^{(m)}$ and $x^{(M)}$ indeed exist and lie in the interior of $\Omega_{0}=[a_{0},b_{0}]$. Since the union of the three intervals is $\Omega_0\setminus \{x^{(m)}, x^{(M)}\}$, necessarily
$$
  \Omega_0^{(1)}=[a_{0},x^{(m)}), \qquad \Omega_0^{(2)}=(x^{(m)},x^{(M)}), \qquad \Omega_0^{(3)}=(x^{(M)},b_{0}].
$$
The following summarizes our discussion.
    
\begin{proposition}\label{pr:monotoneSupport}
  Let $\eps>0$ be small enough such that $\sup_{x\in \Omega_0}|\Sxeps|\leq \frac{1}{3}|\Omega_1|$. There is exactly one $x\in (a_{0},b_{0})$ with $y_{M}(x)=b_1$ and $\xi(x,  y_{M}(x))=0$, denoted $x^{(M)}$, and exactly one $x\in (a_{0},b_{0})$ with $y_{m}(x)=a_1$ and $\xi(x,  y_{m}(x))=0$, denoted $x^{(m)}$. The interval-valued map $x\mapsto \Sxeps=[y_{m}(x),y_{M}(x)]$ is increasing in the sense that 
$$x\mapsto y_{m}(x) \qandq x\mapsto y_{M}(x) \quad\mbox{are increasing.}$$
More precisely, $y_{m}(x)=a_{1}$ for $x\in [a_{0}, x^{(m)}]$ whereas $x\mapsto y_{m}(x)$ is strictly increasing on $(x^{(m)},b_{0}]$ with derivative~\eqref{eq:ymderiv}. Similarly, $y_{M}(x)=b_{1}$ for $x\in [x^{(M)},b_{0}]$ whereas $x\mapsto y_{M}(x)$ is strictly increasing on $[a_{0}, x^{(M)})$ with derivative~\eqref{eq:yMderiv}.
\end{proposition} 
    
    We can now take the derivative in~\eqref{eq:Xi}: Leibniz' rule yields that $f'$ is differentiable on $\Omega_0\setminus \{x^{(m)}, x^{(M)}\}$ with derivative
    \begin{align*}
        \Xi'( x )=    h(y_M(x)) \frac{ f_\varepsilon'(x)-y_M(x)}{x- g_\varepsilon'(y_{M}(x))} \chi_{\Omega_0^{(1)}\cup \Omega_0^{(2)} }(x)-h(y_m(x)) \frac{ f_\varepsilon'(x)-y_m(x)}{x- g_\varepsilon'(y_{m}(x))} \chi_{\Omega_0^{(2)}\cup \Omega_0^{(3)} }(x).
    \end{align*}
Hence, we get
$$ ({\mu_1(\Sxeps)})'=    u_1(y_M(x)) \frac{ f_\varepsilon'(x)-y_M(x)}{x- g_\varepsilon'(y_{M}(x))} \chi_{\Omega_0^{(1)}\cup \Omega_0^{(2)} }(x)-u_1(y_m(x)) \frac{ f_\varepsilon'(x)-y_m(x)}{x- g_\varepsilon'(y_{m}(x))} \chi_{\Omega_0^{(2)}\cup \Omega_0^{(3)} }(x) $$
    and 
    \begin{multline*}
         \left(\int_{\Sxeps} y   u_1( y ) dy \right)'=  u_1(y_M(x)) y_M(x)  \frac{ f_\varepsilon'(x)-y_M(x)}{x- g_\varepsilon'(y_{M}(x))} \chi_{\Omega_0^{(1)}\cup \Omega_0^{(2)} }(x)\\
         - u_1(y_m(x)) y_m(x)\frac{ f_\varepsilon'(x)-y_m(x)}{x- g_\varepsilon'(y_{m}(x))} \chi_{\Omega_0^{(2)}\cup \Omega_0^{(3)} }(x),
    \end{multline*}
 and then
 \begin{align*}
      f_\varepsilon'' ( x )&= \frac{ {\mu_1(\Sxeps)} \left( \int_{\Sxeps} y   u_1( y ) dy \right)'-({\mu_1(\Sxeps)})' \int_{\Sxeps} y   u_1( y ) dy}{({\mu_1(\Sxeps)})^2}\\
      &=\frac{\left( \int_{\Sxeps} y   u_1( y ) dy \right)' -({\mu_1(\Sxeps)})' f'_\varepsilon(y) }{{\mu_1(\Sxeps)} } %
 \end{align*} 
yields the following result. 
\begin{proposition}\label{pr:fDerivativeFormula}
   Let $\eps>0$ be small enough such that $\sup_{x\in \Omega_0}|\Sxeps|\leq \frac{1}{3}|\Omega_1|$. Then using the notation introduced above, 
    \begin{multline}\label{secondderivativeShape}
        f_\varepsilon'' ( x )= u_1(y_m(x)) \frac{ (f_\varepsilon'(x)-y_m(x))^2}{( x-g_\varepsilon'(y_{m}(x)))\mu_1(\Sxeps)}  \chi_{\Omega_0^{(2)}\cup \Omega_0^{(3)}}(x) \\
        + u_1(y_M(x)) \frac{ (f_\varepsilon'(x)-y_M(x))^2}{( g_\varepsilon'(y_{M}(x))-x)\mu_1(\Sxeps)}  \chi_{\Omega_0^{(1)}\cup \Omega_0^{(2)}}(x).
    \end{multline} 
\end{proposition} 

\begin{remark}\label{rk:2ndDerivative}
  Of course there is also a regime (for large $\eps$, due to $\pi_{\eps}\to\mu_{0}\otimes\mu_{1}$ \cite[Remark~2.9]{GonzalezSanzNutz.24}) where $\xi(x,\cdot)$ is strictly positive on $\Omega_{1}$ and in particular $\SXeps{x'}=\Omega_{1}$ for $x'$ close to $x$. Then it is immediate from~\eqref{eq:fprime} that $f''_{\eps}(x)=0$. (In fact, when $\pi_{\eps}$ has full support, the functions $f_{\eps}$ and $g_{\eps}$ are affine and can be found in closed form.) In particular, the condition $\eps\leq\eps_{0}$ is essential to guarantee the strong convexity in \cref{th:Reg}.
\end{remark} 
    
\section{Bounds for the second derivative}\label{se:Bounds2ndDerivative}

Recall our notation
$$ \xi(x,y)=-f_\varepsilon(x)-g_\varepsilon(y)+xy, \qquad \Sxeps=\{y\in\Omega_{1}: \ \xi(x,y)\geq0\} =[y_{m}(x), y_M(x)]. $$ 
In this section, we bound $f_\varepsilon''$ by geometric arguments. First, in \cref{Sec:Comparison}, we relate $|\Sxeps|$ to the maximum of $\xi(x,\cdot)$; i.e., to the maximum of the density of $\pi_{\eps}$. In \cref{subse:dist}, we show that $f'_\varepsilon(x)$ is approximately at the center of the support $\Sxeps$. \Cref{subse:Bounds2ndDerivative} uses these two auxiliary results to prove a crucial intermediate bound relating $|\Sxeps|$ with the minimum and the maximum of $f''_\varepsilon$; cf.~\eqref{eq:boundSx}. Finally, we  deduce the desired bounds on $f_\varepsilon''$.

\subsection{Geometric approach}\label{Sec:Comparison}

For each $ x \in \Omega_0$, the function $\xi( x , \cdot)$ is concave, behaving roughly like a parabola open to the bottom. Our starting point is the following simple observation, relating 
the height $\max_{ y \in \Omega_1}\xi( x ,  y )$ to the size $|\Sxeps|$ of the support of $(\xi( x ,  \cdot ))_{+}$.

\begin{lemma}\label{Lemma:maximunPrincp}
    There exists $C>0$ such that   
    $$ C^{-1} \varepsilon  \leq \max_{ y \in \Omega_1}\xi( x ,  y ) |\Sxeps| \leq C \varepsilon $$
for all $ x \in \Omega_0$ and $\varepsilon>0$.
\end{lemma}

\begin{proof}
Using \eqref{eq:solution} and \cref{as:main},
$$ \varepsilon=\int (\xi( x ,  y ))_+ d\mu_{1}(y)\leq \Lambda \int_{\Omega_1} (\xi( x ,  y ))_+ dy \leq \Lambda |\Sxeps|  \max_{ y \in \Omega_1}\xi( x ,  y ), $$
giving the lower bound. On the other hand, as $y\mapsto \xi( x ,  y )$ is concave, its graph lies above the graph of the piecewise affine function $\ell(y)$ interpolating the three points 
    $$ (y_m(x), \xi( x ,  y_m(x) )), \qquad  \bigg(\argmax_{y\in \Omega_1} \xi( x ,  y ), \max_{y\in \Omega_1}\xi( x ,  y_m(x)) \bigg), \qquad  
    (y_M(x), \xi( x ,  y_M(x) )),$$
which, in turn, lies above the ``tent'' interpolating the points
    $$ (y_m(x), 0 ), \qquad  \bigg(\argmax_{y\in \Omega_1} \xi( x ,  y ), \max_{y\in \Omega_1}\xi( x ,  y_m(x)) \bigg), \qquad 
    (y_M(x), 0),$$
by the definition of $y_m(x),y_M(x)$. The integral of the latter tent corresponds to the area of a triangle, giving
    $$ \varepsilon=\int (\xi( x ,  y ))_+ d\mu_1(y)\geq \lambda \int_{\Omega_1} (\xi( x ,  y ))_+ dy \geq \lambda \int_{\Omega_1} (\ell(y))_+ dy  \geq \frac{\lambda}{2}|\Sxeps|  \max_{ y \in \Omega_1}\xi( x ,  y )$$
after recalling that $|\Sxeps|=y_M(x)-y_m(x)$. This establishes the upper bound.
\end{proof}

In view of \cref{Lemma:maximunPrincp}, our goal is to control the ratio between $\max_{ y \in \Omega_1}\xi( x ,  y )$ and $ \mu_1(\Sxeps)$ (or equivalently $|\Sxeps|$). Clearly this ratio is closely related to the second derivative of $\xi( x ,  \cdot )$, or equivalently $g''_{\eps}$. The following result shows the equivalence of $\max_{ y \in \Omega_1}\xi( x ,  y )$ and $\max_{ x '\in \Omega_0}\xi( x ',f_\varepsilon'( x ))$, which will be important for controlling~$f''_\varepsilon$ and $g''_\varepsilon$ later on.

\begin{lemma}\label{Lemma:comparemax}
There exists $C>0$ such that for every $ x \in \Omega_0$, 
\begin{align*}
  &C^{-1}  \max_{ y \in \Omega_1} \xi( x ,  y ) \,\leq\,   \xi( x , f_\varepsilon'( x ))\,= \,\max_{ x'\in \Omega_0} \xi( x',f_\varepsilon'( x ))\,\leq\, \max_{ y \in \Omega_1} \xi( x , y ).
\end{align*}
\end{lemma}

\begin{proof}
    We can write \eqref{eq:solution} as
    $$ \int_{\Sxeps}xy  d\mu_1( y )-  f_\varepsilon( x ) \mu_1(\Sxeps)  -\int_{\Sxeps}g_\varepsilon( y ) d\mu_1( y )= \varepsilon.    $$
   Therefore, 
     $$  
\frac{\int_{\Sxeps}xy  d\mu_1( y )}{\mu_1(\Sxeps)}
   -f_\varepsilon( x )  - \frac{\int_{\Sxeps}g_\varepsilon( y ) d\mu_1( y )}{\mu_1(\Sxeps)}  = \frac{\varepsilon}{\mu_1(\Sxeps)}.    $$
Recalling from \cref{ContinuityofT} the formula
$f_\varepsilon' ( x )=  \frac{\int_{\Sxeps} y   d\mu_1( y )}{\mu_1(\Sxeps)}, $
we deduce
  $$ x f_\varepsilon'( x )-f_\varepsilon( x )  - \frac{\int_{\Sxeps}g_\varepsilon( y ) d\mu_1( y )}{\mu_1(\Sxeps)}
     = \frac{\varepsilon}{\mu_1(\Sxeps)}  . $$
As  $g_\varepsilon$ is convex,  Jensen's inequality and the aforementioned formula yield
$$ \frac{\varepsilon}{\mu_1(\Sxeps)} \leq - f_\varepsilon( x )  - g_\varepsilon(f_\varepsilon'( x ))+ x f_\varepsilon'( x )  = \xi(x,f_\varepsilon'( x )),
$$
and now the first inequality follows from \cref{Lemma:maximunPrincp}. The second inequality follows trivially, whereas the equality holds by the first-order condition: $y:=f_\varepsilon'( x )$ satisfies $\frac{d}{dx} \xi(x,y)=0$.
\end{proof}

\subsection{Distance of $f'_\varepsilon(x)$ to the boundary of $\Sxeps$} \label{subse:dist}

\Cref{ContinuityofT} shows that $f_\varepsilon'(x)\in \Sxeps$. The following result states that the distance between $f'_\varepsilon(x)$ and either of the two boundary points $y_m(x),y_M(x)$ of $\Sxeps$ is of the same order as $|\Sxeps|=y_M(x)-y_m(x)$, and more precisely that $f'_\varepsilon(x)$ is located in the center half of $\Sxeps$ for small $\varepsilon$.

\begin{lemma}\label{lemma:distanceBound}
    There exists $\varepsilon_0>0$ such that for all $\varepsilon<\varepsilon_0$ and $x\in \Omega_0$,
    \begin{alignat*}{2}
    \frac14 |\Sxeps|&\leq  y_M(x)- f'_\varepsilon(x) && \leq   |\Sxeps|, \\[.2em]
    \frac14 |\Sxeps|&\leq   f'_\varepsilon(x)-y_m(x)&& \leq   |\Sxeps|.
    \end{alignat*}
\end{lemma}

\begin{proof}
The upper bounds are trivial due to $f_\varepsilon'(x)\in \Sxeps$. 
Fix $x\in\Omega_{0}$ and define the mid point of $\Sxeps$ as
    $$m(x):= \frac{\int_{\Sxeps} y    dy}{|\Sxeps|}=\frac{y_M(x)+y_m(x)}{2}.$$
Let $0<\eta<1$. As $u_1$ is uniformly continuous on $\Omega_{1}$ and $u_1\geq \lambda >0$, and as $\sup_{ x  \in \Omega_0}|\Sxeps|\to 0 $ by \cref{pr:qualitative}, there exists $\varepsilon_{0}$ (independent of $x$) such that $\eps\leq \eps_{0}$ implies
\begin{align}\label{eq:densityRatio}
  1-\eta \leq \frac{u_{1}(y)}{u_{1}(y')} \leq 1+\eta \qforallq y,y'\in \Sxeps. 
\end{align} 
For any bounded function $h(y)$, it follows that
\begin{align*}
\left\vert \frac{1}{\mu_1(\Sxeps)}\int_{\Sxeps} h(y)d\mu_1( y )-  \frac{1}{|\Sxeps |}\int_{\Sxeps} h(y) dy \right\vert 
 \leq \frac{\eta}{|\Sxeps |} \int_{\Sxeps} |h(y)| dy.
\end{align*}
Using this with $h(y):=y-m(x)$,
the formula $f_\varepsilon'(x)=\frac{1}{\mu_1 (\Sxeps )}\int_{\Sxeps} y   d\mu_1( y )$ of \cref{ContinuityofT} yields
\begin{align*}
|f'_\varepsilon(x)-m(x)| 
 &=  \left\vert \frac{1}{\mu_1(\Sxeps)}\int_{\Sxeps} (y-m(x))d\mu_1( y )-  \frac{1}{|\Sxeps |}\int_{\Sxeps} (y-m(x)) dy \right\vert \\
 &\leq \frac{\eta}{|\Sxeps |} \int_{\Sxeps} |y-m(x)| dy 
 \leq \frac{\eta}{2|\Sxeps |} \int_{\Sxeps} |\Sxeps| dy  = \frac{\eta}{2} |\Sxeps |.
\end{align*}
(We remark that this conclusion holds for all $\eps>0$, that is with $\eps_{0}=\infty$, if \eqref{eq:densityRatio} holds everywhere on $\Omega_{1}$.) The stated lower bounds are derived via the reverse triangle inequality
\begin{align*}
    f'_\varepsilon(x)-y_m(x)&\geq |m(x)-y_m(x)| -|f'_\varepsilon(x)-m(x)|\\
    &= \frac{|y_M(x)-y_m(x)|}{2} -|f'_\varepsilon(x)-m(x)|
    \geq \frac{1-\eta}{2} |y_M(x)-y_m(x)|
\end{align*}
by taking $\eta=1/2$, and similarly for $y_M(x)$.
\end{proof}

\subsection{Bounds for the second derivative}\label{subse:Bounds2ndDerivative}

We now proceed to bound the second derivative $f''_{\eps}$  in terms of the size of the support. Recall that $\Syeps{y}:=\{ x \in \Omega_0: \ f_\varepsilon( x )+g_\varepsilon( y )-x y \leq 0 \}$. Next, we use the symmetric version of \cref{lemma:distanceBound} with exchanged roles of the marginals: for $\eps<\eps_{0}$ and $y\in\Omega_{1}$,
$$
\frac14|\Syeps{y}|\leq  |g_\varepsilon'(y)-x_{M}(y)| \leq |\Syeps{y}|.
$$
Specifically, we fix $x\in\Omega_{0}$ and use this for $y=y_{M}(x)$. Noting that $x_{M}(y_{M}(x))=x$, we obtain
$$
\frac14|\Syeps{y_{M}(x)}|\leq  |g_\varepsilon'(y_{M}(x))-x| \leq |\Syeps{y_{M}(x)}|.
$$
Similarly,
$$
\frac14|\Syeps{y_{m}(x)}|\leq  |g_\varepsilon'(y_{m}(x))-x| \leq |\Syeps{y_{m}(x)}|.
$$
We now plug these estimates, as well as the estimates stated in \cref{lemma:distanceBound}, into the formula~\eqref{secondderivativeShape} for $f_\varepsilon'' ( x )$. Recalling from \cref{se:SecondOrderDer} that 
\begin{align*}
\mbox{$g_\varepsilon'(y_{M}(x))-x>0$ for all $x\in \Omega_0^{(1)}\cup \Omega_0^{(2)}$},\qquad \mbox{$g_\varepsilon'(y_{m}(x))-x<0$ for all $x\in \Omega_0^{(3)}\cup \Omega_0^{(2)}$,}
\end{align*}
we obtain for all $x\in\Omega_0\setminus \{x^{(m)}, x^{(M)}\}$ that 
\begin{equation}
    \label{secondderivativeUpper}
        f_\varepsilon'' ( x )\leq  C \Lambda \left( \frac{ |\Sxeps|}{\left|\Syeps{y_{m}(x)}\right|}  \chi_{\Omega_0^{(3)}\cup \Omega_0^{(2)} }(x) + \frac{ |\Sxeps|}{\left|\Syeps{y_{M}(x)}\right|}  \chi_{\Omega_0^{(1)}\cup \Omega_0^{(2)} }(x) \right) 
\end{equation}
and 
\begin{equation}
    \label{secondderivativeLower}
        f_\varepsilon'' ( x )\geq  \frac{\lambda}{C}  \left( \frac{ |\Sxeps|}{\left|\Syeps{y_{m}(x)}\right|}  \chi_{\Omega_0^{(3)}\cup \Omega_0^{(2)} }(x) + \frac{ |\Sxeps|}{\left|\Syeps{y_{M}(x)}\right|}  \chi_{\Omega_0^{(1)}\cup \Omega_0^{(2)} }(x) \right). 
\end{equation}
From these two bounds and \cref{Lemma:maximunPrincp}, we see that
\begin{align*}
\sigma_m(f_\varepsilon)&:= \inf_{x\in \Omega_0\setminus \{x^{(m)}, x^{(M)}\}} f_\varepsilon'' ( x ) >0,\\\
\sigma_M(f_\varepsilon)&:= \sup_{x\in \Omega_0\setminus \{x^{(m)}, x^{(M)}\}} f_\varepsilon'' ( x ) <+\infty.
\end{align*}
We define $\sigma_m(g_\varepsilon)$ and  $\sigma_M(g_\varepsilon)$ in a  symmetric manner. Next, we prove that these four quantities are upper and lower bounded uniformly in $\eps$. Recall from~\eqref{eq:solution2} that
\begin{align}
        \label{eq:calcInteg}\varepsilon&=\int (\xi( x ',  f_\varepsilon' ( x )))_+ d\mu_0(x')\geq \lambda  \int_{a_0}^{b_0} (\xi( x ',  f_\varepsilon' ( x )))_+  dx'.
\end{align}
Define the convex conjugate $ f_\varepsilon^*(y):=\sup_{x\in \Omega_0} \left\{   x y  - f_\varepsilon(x) \right\} .$
A second order Taylor development of $f_\varepsilon( x ')$ around $ x $ gives 
\begin{align}\label{eq:Taylor2ndOrder}
 \xi( x ',  f_\varepsilon' ( x ))&=-f_\varepsilon( x ' )-g_\varepsilon( f_\varepsilon' ( x ))+  x '   f_\varepsilon' ( x ) 
 \notag \\
 &=[f_\varepsilon^*( f_\varepsilon' ( x ))-g_\varepsilon( f_\varepsilon' ( x ))]-[f_\varepsilon( x ')+f_\varepsilon^*( f_\varepsilon' ( x ))- x '  f_\varepsilon' ( x ) ] \notag\\
 &\geq [f_\varepsilon^*( f_\varepsilon' ( x ))-g_\varepsilon( f_\varepsilon' ( x ))]- \frac{ \sigma_M(f_\varepsilon) ( x - x ')^2}{2}.
 \end{align} 
 Call $t=[f_\varepsilon^*( f_\varepsilon' ( x ))-g_\varepsilon( f_\varepsilon' ( x ))]$.
As above, the qualitative convergence result \cref{pr:qualitative} yields $\varepsilon_0>0$ such that  $\sup_{y\in \Omega_1}|\Syeps{y}|\leq \frac{1}{2}|\Omega_1|$ for $\varepsilon<\varepsilon_0$, and hence at least one of the conditions $a_0\notin \Syeps{y}$ or $b_0\notin \Syeps{y}$ must hold for any $y$. Assume w.l.o.g.\ that $b_0\notin \Syeps{y}$. Returning to~\eqref{eq:calcInteg} and applying~\eqref{eq:Taylor2ndOrder} to the integrand, we deduce
\begin{align*}
    \varepsilon&\geq \lambda  \int_{x}^{\infty} (\xi( x ',  f_\varepsilon' ( x )))_+  dx'\\
    &\geq \lambda  \int_{x}^{\infty}\left( t- \frac{ \sigma_M(f_\varepsilon) ( x - x ')^2}{2}\right)_+ dx'\\
    &= \lambda  \sigma_M(f_\varepsilon)\int_{0}^{\infty}\left( \frac{t}{\sigma_M(f_\varepsilon)}- \frac{  u^2}{2}\right)_+ d u\\
    &= \frac{2\sqrt{2}}{3} \lambda  \sigma_M(f_\varepsilon) \left(\frac{t}{\sigma_M(f_\varepsilon)}\right)^{3/2}= \frac{2\sqrt{2}}{3}\lambda  (\sigma_M(f_\varepsilon))^{-1/2} t^{3/2}.
\end{align*}
By the same means,
$ \varepsilon\leq \frac{4\sqrt{2}}{3}\Lambda  (\sigma_m(f_\varepsilon))^{-1/2} t^{3/2} $. In summary, for all $\varepsilon<\varepsilon_0$ and $x\in \Omega_0$,
\begin{align*}\frac{2\sqrt{2}}{3}\lambda  (\sigma_M(f_\varepsilon))^{-1/2} t^{3/2}
\leq \varepsilon 
\leq \frac{4\sqrt{2}}{3}\Lambda  (\sigma_m(f_\varepsilon))^{-1/2} t^{3/2}.
\end{align*}
Here
$t=f_\varepsilon^*( f_\varepsilon' ( x ))-g_\varepsilon( f_\varepsilon' ( x )) = \max_{x'\in \Omega_0} \xi( x ',  f_\varepsilon' ( x ))$ due to $\partial_{x'} \xi( x ',  f_\varepsilon' ( x ))= f_\varepsilon' ( x )-f_\varepsilon' ( x' )$. 
Recalling also from \cref{Lemma:comparemax} that 
\[C^{-1}  \max_{ y \in \Omega_1} \xi( x ,  y )\leq  \max_{ x '\in \Omega_0} \xi( x ',f_\varepsilon'( x ))\leq \max_{ y \in \Omega_1} \xi( x , y ) \qforallq x\in \Omega_0,\]
and allowing $C$ to vary from line to line, we deduce
$$
    C^{-1}   (\sigma_M(f_\varepsilon))^{-1/2} \max_{ y \in \Omega_1} (\xi( x ,  y ))_+^{3/2} \leq \varepsilon\leq C (\sigma_m(f_\varepsilon))^{-1/2} \max_{ y \in \Omega_1} (\xi( x ,  y ))_+^{3/2}.
$$
In view of \cref{Lemma:maximunPrincp}, this implies
$$
      C^{-1}   (\sigma_M(f_\varepsilon))^{-1/2} |\Sxeps|^{-3/2} \leq \varepsilon^{-1/2}\leq C (\sigma_m(f_\varepsilon))^{-1/2} |\Sxeps|^{-3/2},
$$
and after rearranging we conclude
 \begin{equation}\label{eq:boundSx} C^{-1}\left(\frac{\varepsilon}{\sigma_M(f_\varepsilon)}\right)^{\frac{1}{3}} \leq  |\Sxeps| \leq C\left(\frac{\varepsilon}{\sigma_m(f_\varepsilon)}\right)^{\frac{1}{3}}. 
 \end{equation}
We observe that \cref{th:rate} will follow from \eqref{eq:boundSx} once uniform bounds for $\sigma_M(f_\varepsilon)$ and $\sigma_m(f_\varepsilon)$ are derived. To established the latter, we plug \eqref{eq:boundSx} into \eqref{secondderivativeUpper} and \eqref{secondderivativeLower} to obtain the following bounds for  all $x\notin \{x^{(m)}, x^{(M)}\}$:
$$ f_\varepsilon'' ( x )\leq  C \left(\frac{\sigma_M(g_\varepsilon)}{\sigma_m(f_\varepsilon)}\right)^{\frac{1}{3}} \quad {\rm and}\quad  f_\varepsilon'' (x)\geq  \frac{1}{C} \left(\frac{\sigma_m(g_\varepsilon)}{\sigma_M(f_\varepsilon)}\right)^{\frac{1}{3}} . $$
By taking the supremum in the first expression and the infimum in the second, we deduce
\begin{align}\label{eq:sigma1}
\sigma_M(f_\varepsilon)\leq  C \left(\frac{\sigma_M(g_\varepsilon)}{\sigma_m(f_\varepsilon)}\right)^{\frac{1}{3}} \quad {\rm and}\quad \sigma_m(f_\varepsilon)\geq  \frac{1}{C} \left(\frac{\sigma_m(g_\varepsilon)}{\sigma_M(f_\varepsilon)}\right)^{\frac{1}{3}}  .
\end{align} 
Symmetrically,
\begin{align}\label{eq:sigma2}
\sigma_M(g_\varepsilon)\leq  C \left(\frac{\sigma_M(f_\varepsilon)}{\sigma_m(g_\varepsilon)}\right)^{\frac{1}{3}} \quad {\rm and}\quad \sigma_m(g_\varepsilon)\geq  \frac{1}{C} \left(\frac{\sigma_m(f_\varepsilon)}{\sigma_M(g_\varepsilon)}\right)^{\frac{1}{3}}  .
\end{align}
Plugging~\eqref{eq:sigma2} into~\eqref{eq:sigma1} yields 
$$ \sigma_M(f_\varepsilon)\leq  C \left(\frac{ \frac{\sigma_M(f_\varepsilon)}{\sigma_m(g_\varepsilon)}}{\frac{\sigma_m(g_\varepsilon)}{\sigma_M(f_\varepsilon)} }\right)^{\frac{1}{9}}=C \left( \frac{\sigma_M(f_\varepsilon)}{\sigma_m(g_\varepsilon)}\right)^{\frac{2}{9}} $$
and 
$$ \sigma_m(g_\varepsilon)\geq  \frac{1}{C} \left(\frac{\frac{\sigma_m(g_\varepsilon)}{\sigma_M(f_\varepsilon)}}{\frac{\sigma_M(f_\varepsilon)}{\sigma_m(g_\varepsilon)} }\right)^{\frac{1}{9}}= \frac{1}{C} \left(\frac{\sigma_m(g_\varepsilon)}{\sigma_M(f_\varepsilon)}\right)^{\frac{2}{9}},$$
which implies
$$ \sigma_M(f_\varepsilon)\leq  C \left( \frac{1}{\sigma_m(g_\varepsilon)}\right)^{\frac{2}{7}}\quad {\rm and}
\quad \sigma_m(g_\varepsilon)\geq   \frac{1}{C} \left(\frac{1}{\sigma_M(f_\varepsilon)}\right)^{\frac{2}{7}}$$
and finally
$$ \sigma_M(f_\varepsilon)\leq  C \sigma_M(f_\varepsilon)^{\frac{4}{49}}
\quad {\rm and}
\quad \sigma_m(g_\varepsilon)\geq   C^{-1}  \sigma_m(g_\varepsilon)^{\frac{4}{49}}$$
for some $C>0$. It follows that $\sigma_M(f_\varepsilon)$ and $\sigma_m(g_\varepsilon)$ are uniformly upper and lower bounded, respectively. Symmetric statements hold for $\sigma_M(g_\varepsilon)$ and $\sigma_m(f_\varepsilon)$, completing the proof of \cref{th:Reg}.  Moreover, \cref{th:rate} now follows from \eqref{eq:boundSx}.

\section{Convergence to unregularized optimal transport}\label{section:ratesPotentials}

The goal of this section is to prove \cref{th:ratesDer} and its corollary. We first prove a rate for the barycentric projection of $\pi_\varepsilon$.

\begin{lemma}\label{le:barycentricProjRate}
    Let $S_\varepsilon(x)$ be the barycentric projection of $\pi_\varepsilon$; that is,
$$ S_\varepsilon(x) : = \int y \left( \frac{\xi( x ,  y )}{\varepsilon}\right)_+ d\mu_1(y). $$
    There exists $C>0$ such that for all $\eps\in (0,1]$,
    $$\|T_0-S_\varepsilon\|_{L^2(\Omega_0,dx)}\leq C 
    \varepsilon^{\frac{1}{3}}.$$
\end{lemma}

The argument follows \cite[Section~4.2]{CarlierPegonTamanini.22} where a similar result is shown in the context of entropic optimal transport.

\begin{proof}
    Recall that $(f_{0},g_{0})$ denote the potentials for the unregularized optimal transport problem, which are convex conjugates of one another. 
Since 
$ g_0 $
is convex and $ \eps^{-1}(\xi( x ,  y ))_+ d \mu_1( y )$ is a probability measure for each $x\in \Omega_0$ by~\eqref{eq:solution},  Jensen's inequality implies that
$$ \int  g_0( y ) \left( \frac{\xi( x ,  y )}{\varepsilon}\right)_+ d \mu_1( y ) \geq g_0\left(\int  y \left( \frac{\xi( x ,  y )}{\varepsilon}\right)_+ d \mu_1( y ) \right)=g_0\left(S_\varepsilon(x)\right). $$
This yields the inequality in 
\begin{align*}
    0&= \int g_0( y ) d \pi_\varepsilon( x , y ) -  \int g_0( y ) d \mu_1( y ) \\
    &= \int \left(\int  g_0( y ) \left( \frac{\xi( x ,  y )}{\varepsilon}\right)_+ d \mu_1( y ) \right) d\mu_0( x )  -   \int g_0(T_0 ( x)) d \mu_0( x ) \\
    & \geq \int g_0\left(S_\varepsilon(x) \right) d\mu_0( x )  -   \int g_0(T_0 ( x)) d \mu_0( x ) .
\end{align*}
As $g_0$ is strongly convex,
$$
  g_0\left(S_\varepsilon(x) \right) - g_0(T_0 ( x)) \geq (S_\varepsilon(x)-T_0 ( x))g_{0}'(T_{0}(x)) + C^{-1} (S_\varepsilon(x)-T_0 ( x))^{2}
$$
for some $C>0$ and we deduce
$$0\geq  \int     g_0'(T_0( x ) ) \left(S_\varepsilon(x)-  T_0( x ) \right) d \mu_0( x )
     +\lambda C^{-1}\left\| S_\varepsilon -  T_0 \right\|_{L^2(\Omega_0,dx)}^2. $$
In view of $g_0'(T_0( x ))=x $, it follows that
\begin{align*}
    \lambda C^{-1}\left\| S_\varepsilon(x)-  T_0 \right\|_{L^2(\Omega_0,dx)}^2
    & \leq \int    \left(xT_0( x ) -xS_\varepsilon(x) \right) d \mu_0( x )\\
    & =  \int  x T_0 ( x ) d \mu_0( x )- \int    x \int  y  \left( \frac{\xi( x ,  y )}{\varepsilon}\right)_+ d \mu_1( y )  d \mu_0( x )\\
    & =  \int     x  T_0 ( x ) d \mu_0(x )- \int   xy \, d \pi_\varepsilon( x ,  y ) \\
    &  =  \frac{1}{2} \int   (  x - y )^2 d \pi_\varepsilon( x ,  y )-\frac{1}{2} \int (x-T_0(x))^2 d \mu_0(x).
\end{align*}
The last line is the difference between the transport cost of $\pi_{\eps}$ and the optimal transport cost. As the quadratic penalty is nonnegative, that difference is dominated by the difference between the optimal regularized transport cost and the optimal transport cost. By \cite[Corollary~3.14]{EcksteinNutz.22}, the latter satisfies
\begin{equation*}
     \frac{1}{2}\int (x-y)^2 d \pi_{\varepsilon}(x,y)-  \frac{1}{2}\int (x-T_0(x))^2 d \mu_0(x) + {\varepsilon}\left\| \frac{d\pi_\varepsilon}{d(\mu_0\otimes \mu_1)}\right\|_{L^2(\mu_0\otimes \mu_1)}^2  \leq C\varepsilon^{\frac{2}{3}}
\end{equation*}
for all $\eps\in (0,1]$, for some $C>0$. (This rate is sharp by \cite[Proposition~4.4]{EcksteinNutz.22}.) As a result, $\left\| S_\varepsilon(x)-  T_0 \right\|_{L^2(\Omega_0,dx)}^2 \leq C\varepsilon^{\frac{2}{3}}$, which was the claim.
\end{proof}

\begin{proof}[Proof of \cref{th:ratesDer}.]
By \cref{ContinuityofT} we have
$f_\varepsilon' ( x )\in {\Sxeps}$ for all $x\in \Omega_0$. Thus
$$
  \|f'_\varepsilon(x)-T_0(x)\|_{L^2(\Omega_0,dx)} \leq \left\|d\left(\Sxeps , T_0(x)\right)\right\|_{L^{2}(\Omega_{0},dx)}
$$
and hence the second claim in \cref{th:ratesDer} follows from the first claim, namely that 
    $$ {C}^{-1}{\varepsilon^{\frac{1}{3}}}\leq \left\|d\left(\Sxeps , T_0(x)\right)\right\|_{L^{2}(\Omega_{0},dx)} \leq C\varepsilon^{\frac{1}{3}}.$$
Note that the lower bound already follows from \cref{th:rate}. Whereas for the upper bound, in view of \cref{th:rate}, it suffices to exhibit measurable functions $\{S_\varepsilon\}_{\varepsilon>0}$ such that $S_\varepsilon(x)\in \Sxeps$  for all $x\in \Omega_0$ and $\|S_\varepsilon-T_0\|_{L^2(\Omega_0,dx)}\leq C\varepsilon^{\frac{1}{3}}$ for all $\varepsilon\leq\eps_{0}$. By \cref{le:barycentricProjRate}, the barycentric projection has these properties.
\end{proof}

\begin{proof}[Proof of \cref{co:potentialRate}]
We denote by $C > 0$ a constant
that may change from line to line, but is uniform in~$\eps$.
By \eqref{eq:rateDerivative} we have
  \begin{equation}
      \label{eq:rateDerivativeAgain}
      \|f'_\varepsilon-f'_0\|_{L^2(\Omega_0,dx)} \leq C\varepsilon^{\frac{1}{3}}.
  \end{equation}
  Let 
  $$\alpha_{\eps} := \frac{1}{|\Omega_{0}|}\int_{\Omega_{0}} (f_\varepsilon(x)-f_{0}(x))dx.$$ Poincar\'e's inequality \cite[Theorem~1, p.\,275]{Evans.10} applied to~\eqref{eq:rateDerivativeAgain} shows
  \begin{equation}
      \label{eq:PoincareStep}
  \|f_\varepsilon-f_{0} - \alpha_{\eps}\|_{L^2(\Omega_0,dx)} \leq C\|f'_\varepsilon-f'_0\|_{L^2(\Omega_0,dx)} \leq C\varepsilon^{\frac{1}{3}}.
  \end{equation}
 Let 
  $$
    \beta_{\eps} := \int (f_\varepsilon(x)-f_{0}(x)) \mu_{0}(x).
  $$
  The mean $\alpha_{\eps}$ under the Lebesgue measure is not directly comparable to the mean $\beta_{\eps}$ under~$\mu_{0}$. However, using the variational characterization of the mean and \cref{as:main},
  \begin{align*}
    \|f_\eps-f_{0} - \beta_{\eps}\|_{L^{2}(\mu_{0})} = \inf_{\alpha\in\R} \|f_\eps-f_{0} - \alpha\|_{L^{2}(\mu_{0})} \leq C \inf_{\alpha\in\R} \|f_\eps-f_{0} - \alpha\|_{L^{2}(\Omega_0,dx)}
  \end{align*}
  which in conjunction with~\eqref{eq:PoincareStep} yields
  \begin{equation}
      \label{eq:PoincareStep2}
  \|f_\varepsilon-f_{0} - \beta_{\eps}\|_{L^2(\mu_{0})} \leq C\varepsilon^{\frac{1}{3}}.
  \end{equation}
  We argue below that 
  \begin{equation}
      \label{eq:rateExpectationPotential}
      |\beta_{\eps}| \leq C\varepsilon^{\frac{1}{3}}.
  \end{equation}  
  Then, \eqref{eq:PoincareStep2} implies
  $
  \|f_\varepsilon-f_{0}\|_{L^2(\mu_{0})} \leq C\varepsilon^{\frac{1}{3}},
  $  
  and thus   $
  \|f_\varepsilon-f_{0}\|_{L^2(\Omega_0,dx)} \leq C\varepsilon^{\frac{1}{3}}
  $
  by \cref{as:main}.
  Together with \eqref{eq:rateDerivativeAgain}, we now have
  $$
  \|f_\varepsilon-f_{0}\|_{W^{1,2}(\Omega_0)} \leq C\varepsilon^{\frac{1}{3}},
  $$ 
  so that Morrey's inequality \cite[Theorem~4, p.\,266]{Evans.10} yields the claim.  
  
	It remains to show~\eqref{eq:rateExpectationPotential}.
In fact, we prove the stronger claim
  \begin{align}\label{eq:expectClaim}
     2|\beta_{\eps}| = \left| \int (f_\varepsilon(x)- f_{0}(x)) \mu_{0}(x) + \int (g_\varepsilon(y) -g_0(y)) \mu_{1}(y) \right| \leq C\varepsilon^{\frac{2}{3}}
  \end{align} 
  where the equality is due to~\eqref{eq:normalizationPot}.
  Since $\pi_{0}$ is optimal for the transport problem, \eqref{qotIntro} implies 
\begin{equation}\label{eq:penaltyBound}
  \frac{\varepsilon}{2}\left\| \frac{d \pi_{\eps}}{ d(\mu_0\otimes \mu_1)}\right\|^2_{L^{2}(\mu_0\otimes \mu_1)} \leq \QOT_\eps - \OT.
\end{equation}
Next, we use the strong duality of $\QOT_\eps$ (e.g., \cite[Theorem~2.2]{Nutz.24}). The dual expression~\eqref{dual} for $\QOT_\eps$ (taking into account the change of variables~\eqref{ex:reparam}) and the analogous dual of $\OT$ yield
\begin{align*}
  \QOT_\eps - \OT 
  &= \int (f_0(x)- f_{\eps}(x)) \mu_{0}(x) + \int (g_0(y) -g_{\eps}(y)) \mu_{1}(y)  - \frac{\varepsilon}{2}\left\| \frac{d \pi_{\eps}}{ d(\mu_0\otimes \mu_1)}\right\|^2_{L^{2}(\mu_0\otimes \mu_1)}.
\end{align*}
In view of~\eqref{eq:penaltyBound}, this implies
\begin{align*}
  \QOT_\eps - \OT 
  &\leq \int (f_0(x)- f_{\eps}(x)) \mu_{0}(x) + \int (g_0(y) -g_{\eps}(y)) \mu_{1}(y)
   \leq 2(\QOT_\eps - \OT).
\end{align*}
The claim~\eqref{eq:expectClaim} follows as  $C^{-1}\varepsilon^{\frac{2}{3}} \leq  \QOT_\eps - \OT\leq C\varepsilon^{\frac{2}{3}}$ by \cite[Corollary~3.14 and Proposition~4.4]{EcksteinNutz.22}.
\end{proof}

\bibliographystyle{abbrv}
\bibliography{stochfin}

\newcommand{\dummy}[1]{}
\begin{thebibliography}{10}

\bibitem{AltschulerNilesWeedStromme.21}
J.~M. Altschuler, J.~Niles-Weed, and A.~J. Stromme.
\newblock Asymptotics for semidiscrete entropic optimal transport.
\newblock {\em SIAM J. Math. Anal.}, 54(2):1718--1741, 2022.

\bibitem{BayraktarEcksteinZhang.22}
E.~Bayraktar, S.~Eckstein, and X.~Zhang.
\newblock Stability and sample complexity of divergence regularized optimal transport.
\newblock {\em Bernoulli}, 31(1):213--239, 2025.

\bibitem{Berman.20}
R.~J. Berman.
\newblock The {S}inkhorn algorithm, parabolic optimal transport and geometric {M}onge-{A}mp{\`e}re equations.
\newblock {\em Numer. Math.}, 145(4):771--836, 2020.

\bibitem{BerntonGhosalNutz.21}
E.~Bernton, P.~Ghosal, and M.~Nutz.
\newblock Entropic optimal transport: {G}eometry and large deviations.
\newblock {\em Duke Math. J.}, 171(16):3363--3400, 2022.

\bibitem{blondel18quadratic}
M.~Blondel, V.~Seguy, and A.~Rolet.
\newblock Smooth and sparse optimal transport.
\newblock volume~84 of {\em Proceedings of Machine Learning Research}, pages 880--889, 2018.

\bibitem{Caffarelli.96}
L.~A. Caffarelli.
\newblock Boundary regularity of maps with convex potentials. {II}.
\newblock {\em Ann. of Math. (2)}, 144(3):453--496, 1996.

\bibitem{CarlierDuvalPeyreSchmitzer.17}
G.~Carlier, V.~Duval, G.~Peyr\'{e}, and B.~Schmitzer.
\newblock Convergence of entropic schemes for optimal transport and gradient flows.
\newblock {\em SIAM J. Math. Anal.}, 49(2):1385--1418, 2017.

\bibitem{CarlierPegonTamanini.22}
G.~Carlier, P.~Pegon, and L.~Tamanini.
\newblock Convergence rate of general entropic optimal transport costs.
\newblock {\em Calc. Var. Partial Differential Equations}, 62(4):Paper No. 116, 28, 2023.

\bibitem{ChenLiuWang.21}
S.~Chen, J.~Liu, and X.-J. Wang.
\newblock Global regularity for the {M}onge-{A}mp\`ere equation with natural boundary condition.
\newblock {\em Ann. of Math. (2)}, 194(3):745--793, 2021.

\bibitem{ChenGeorgiouPavon.16}
Y.~Chen, T.~T. Georgiou, and M.~Pavon.
\newblock On the relation between optimal transport and {S}chr\"{o}dinger bridges: a stochastic control viewpoint.
\newblock {\em J. Optim. Theory Appl.}, 169(2):671--691, 2016.

\bibitem{ChewiPooladian.23}
S.~Chewi and A.-A. Pooladian.
\newblock An entropic generalization of {C}affarelli's contraction theorem via covariance inequalities.
\newblock {\em C. R. Math. Acad. Sci. Paris}, 361:1471--1483, 2023.

\bibitem{CominettiSanMartin.94}
R.~Cominetti and J.~San~Mart\'{\i}n.
\newblock Asymptotic analysis of the exponential penalty trajectory in linear programming.
\newblock {\em Math. Programming}, 67(2, Ser. A):169--187, 1994.

\bibitem{ConfortiTamanini.19}
G.~Conforti and L.~Tamanini.
\newblock A formula for the time derivative of the entropic cost and applications.
\newblock {\em J. Funct. Anal.}, 280(11):108964, 2021.

\bibitem{Cuturi.13}
M.~Cuturi.
\newblock Sinkhorn distances: Lightspeed computation of optimal transport.
\newblock In {\em Advances in Neural Information Processing Systems 26}, pages 2292--2300. 2013.

\bibitem{delBarrioGonzalezLoubes.24}
E.~del Barrio, A.~Gonz\'alez-Sanz, and J.-M. Loubes.
\newblock Central limit theorems for general transportation costs.
\newblock {\em Ann. Inst. Henri Poincar\'e{} Probab. Stat.}, 60(2):847--873, 2024.

\bibitem{DesseinPapadakisRouas.18}
A.~Dessein, N.~Papadakis, and J.-L. Rouas.
\newblock Regularized optimal transport and the rot mover's distance.
\newblock {\em J. Mach. Learn. Res.}, 19(15):1--53, 2018.

\bibitem{EcksteinKupper.21}
S.~Eckstein and M.~Kupper.
\newblock Computation of optimal transport and related hedging problems via penalization and neural networks.
\newblock {\em Appl. Math. Optim.}, 83(2):639--667, 2021.

\bibitem{EcksteinNutz.22}
S.~Eckstein and M.~Nutz.
\newblock Convergence rates for regularized optimal transport via quantization.
\newblock {\em Math. Oper. Res.}, 49(2):1223--1240, 2024.

\bibitem{EssidSolomon.18}
M.~Essid and J.~Solomon.
\newblock Quadratically regularized optimal transport on graphs.
\newblock {\em SIAM J. Sci. Comput.}, 40(4):A1961--A1986, 2018.

\bibitem{Evans.10}
L.~C. Evans.
\newblock {\em Partial differential equations}, volume~19 of {\em Graduate Studies in Mathematics}.
\newblock American Mathematical Society, Providence, RI, second edition, 2010.

\bibitem{GarrizMolinaGonzalezSanzMordant.24}
A.~Garriz-Molina, A.~Gonz{\'a}lez-Sanz, and G.~Mordant.
\newblock Infinitesimal behavior of quadratically regularized optimal transport and its relation with the porous medium equation.
\newblock {\em Preprint arXiv:2407.21528v1}, 2024.

\bibitem{genevay2019sample}
A.~Genevay, L.~Chizat, F.~Bach, M.~Cuturi, and G.~Peyr{\'e}.
\newblock Sample complexity of {S}inkhorn divergences.
\newblock In {\em The 22nd International Conference on Artificial Intelligence and Statistics}, pages 1574--1583. PMLR, 2019.

\bibitem{GeneveyEtAl.16}
A.~Genevay, M.~Cuturi, G.~Peyr\'{e}, and F.~Bach.
\newblock Stochastic optimization for large-scale optimal transport.
\newblock In {\em Advances in Neural Information Processing Systems 29}, pages 3440--3448, 2016.

\bibitem{GhosalNutzBernton.21b}
P.~Ghosal, M.~Nutz, and E.~Bernton.
\newblock Stability of entropic optimal transport and {S}chr\"{o}dinger bridges.
\newblock {\em J. Funct. Anal.}, 283(9):Paper No. 109622, 2022.

\bibitem{GonzalezSanzDelBarrioNutz.25}
A.~Gonz{\'a}lez-Sanz, E.~del Barrio, and M.~Nutz.
\newblock Sample complexity of quadratically regularized optimal transport.
\newblock {\em Preprint arXiv:2511.09807}, 2025.

\bibitem{GonzalezSanzNutz.24}
A.~Gonz\'alez-Sanz and M.~Nutz.
\newblock Quantitative convergence of quadratically regularized linear programs.
\newblock {\em Appl. Math. Optim.}, 91(3):Paper No. 68, 2025.

\bibitem{GonzalezSanzNutzRiveros.25}
A.~Gonz{\'a}lez-Sanz, M.~Nutz, and A.~Riveros~Valdevenito.
\newblock Linear convergence of gradient descent for quadratically regularized optimal transport.
\newblock {\em Preprint arXiv:2509.08547}, 2025.

\bibitem{GonzalezSanzNutzRiveros.24}
A.~Gonz\'alez-Sanz, M.~Nutz, and A.~Riveros~Valdevenito.
\newblock Monotonicity in quadratically regularized linear programs.
\newblock {\em SIAM J. Optim.}, 35(2):1419--1437, 2025.

\bibitem{GonzalezSanzNutzRiveros.26}
A.~Gonz{\'a}lez-Sanz, M.~Nutz, and A.~Riveros~Valdevenito.
\newblock Polyak{--}{{\L{}}}ojasiewicz inequality for quadratically regularized optimal transport.
\newblock {\em Preprint}, 2026.

\bibitem{GulrajaniAhmedArjovskyDumoulinCourville.17}
I.~Gulrajani, F.~Ahmed, M.~Arjovsky, V.~Dumoulin, and A.~Courville.
\newblock Improved training of {W}asserstein {GAN}s.
\newblock In {\em Proceedings of the 31st International Conference on Neural Information Processing Systems}, pages 5769--5779, 2017.

\bibitem{Leonard.12}
C.~L\'{e}onard.
\newblock From the {S}chr\"{o}dinger problem to the {M}onge-{K}antorovich problem.
\newblock {\em J. Funct. Anal.}, 262(4):1879--1920, 2012.

\bibitem{LiGenevayYurochkinSolomon.20}
L.~Li, A.~Genevay, M.~Yurochkin, and J.~Solomon.
\newblock Continuous regularized {W}asserstein barycenters.
\newblock In H.~Larochelle, M.~Ranzato, R.~Hadsell, M.~Balcan, and H.~Lin, editors, {\em Advances in Neural Information Processing Systems}, volume~33, pages 17755--17765. Curran Associates, Inc., 2020.

\bibitem{LorenzMahler.22}
D.~Lorenz and H.~Mahler.
\newblock Orlicz space regularization of continuous optimal transport problems.
\newblock {\em Appl. Math. Optim.}, 85(2):Paper No. 14, 33, 2022.

\bibitem{LorenzMannsMeyer.21}
D.~Lorenz, P.~Manns, and C.~Meyer.
\newblock Quadratically regularized optimal transport.
\newblock {\em Appl. Math. Optim.}, 83(3):1919--1949, 2021.

\bibitem{MangasarianMeyer.79}
O.~L. Mangasarian and R.~R. Meyer.
\newblock Nonlinear perturbation of linear programs.
\newblock {\em SIAM J. Control Optim.}, 17(6):745--752, 1979.

\bibitem{MenaWeed.19}
G.~Mena and J.~Niles-Weed.
\newblock Statistical bounds for entropic optimal transport: sample complexity and the central limit theorem.
\newblock In {\em Advances in Neural Information Processing Systems 32}, pages 4541--4551. 2019.

\bibitem{Mikami.02}
T.~Mikami.
\newblock Optimal control for absolutely continuous stochastic processes and the mass transportation problem.
\newblock {\em Electron. Comm. Probab.}, 7:199--213, 2002.

\bibitem{Mikami.04}
T.~Mikami.
\newblock Monge's problem with a quadratic cost by the zero-noise limit of {$h$}-path processes.
\newblock {\em Probab. Theory Related Fields}, 129(2):245--260, 2004.

\bibitem{Nutz.20}
M.~Nutz.
\newblock {\em Introduction to Entropic Optimal Transport}.
\newblock Lecture notes, Columbia University, 2021.
\newblock \url{https://www.math.columbia.edu/~mnutz/docs/EOT_lecture_notes.pdf}.

\bibitem{Nutz.24}
M.~Nutz.
\newblock Quadratically regularized optimal transport: existence and multiplicity of potentials.
\newblock {\em SIAM J. Math. Anal.}, 57(3):2622--2649, 2025.

\bibitem{NutzWiesel.21}
M.~Nutz and J.~Wiesel.
\newblock Entropic optimal transport: convergence of potentials.
\newblock {\em Probab. Theory Related Fields}, 184(1-2):401--424, 2022.

\bibitem{Pal.19}
S.~Pal.
\newblock On the difference between entropic cost and the optimal transport cost.
\newblock {\em Ann. Appl. Probab.}, 34(1B):1003--1028, 2024.

\bibitem{PeyreCuturi.19}
G.~Peyr{\'e} and M.~Cuturi.
\newblock Computational optimal transport: With applications to data science.
\newblock {\em Foundations and Trends in Machine Learning}, 11(5-6):355--607, 2019.

\bibitem{Rockafellar.97}
R.~T. Rockafellar.
\newblock {\em Convex analysis}.
\newblock Princeton Landmarks in Mathematics. Princeton University Press, Princeton, NJ, 1997.
\newblock Reprint of the 1970 original, Princeton Paperbacks.

\bibitem{RockafellarWets.98}
R.~T. Rockafellar and R.~J.-B. Wets.
\newblock {\em Variational analysis}, volume 317 of {\em Grundlehren der Mathematischen Wissenschaften}.
\newblock Springer-Verlag, Berlin, 1998.

\bibitem{seguy2018large}
V.~Seguy, B.~B. Damodaran, R.~Flamary, N.~Courty, A.~Rolet, and M.~Blondel.
\newblock Large scale optimal transport and mapping estimation.
\newblock In {\em International Conference on Learning Representations}, 2018.

\bibitem{Urbas.97}
J.~Urbas.
\newblock On the second boundary value problem for equations of {M}onge-{A}mp\`ere type.
\newblock {\em J. Reine Angew. Math.}, 487:115--124, 1997.

\bibitem{Vazquez.07}
J.~L. V{\'a}zquez.
\newblock {\em The porous medium equation}.
\newblock Oxford Mathematical Monographs. The Clarendon Press, Oxford University Press, Oxford, 2007.

\bibitem{Villani.09}
C.~Villani.
\newblock {\em Optimal transport, old and new}, volume 338 of {\em Grundlehren der Mathematischen Wissenschaften}.
\newblock Springer-Verlag, Berlin, 2009.

\bibitem{WieselXu.24}
J.~Wiesel and X.~Xu.
\newblock Sparsity of {Q}uadratically {R}egularized {O}ptimal {T}ransport: {B}ounds on {C}oncentration and {B}ias.
\newblock {\em SIAM J. Math. Anal.}, 57(6):6498--6521, 2025.

\bibitem{ZhangMordantMatsumotoSchiebinger.23}
S.~Zhang, G.~Mordant, T.~Matsumoto, and G.~Schiebinger.
\newblock Manifold learning with sparse regularised optimal transport.
\newblock {\em Preprint arXiv:2307.09816v1}, 2023.

\end{thebibliography}
\end{document}